\documentclass[12pt]{article}
\usepackage{graphicx}
\usepackage{amsmath}
\usepackage{amssymb}
\usepackage{amsthm}
\usepackage{mathrsfs}
\usepackage{enumitem}
\usepackage{subcaption} 

\usepackage{pgfplots}
\pgfplotsset{compat=1.8}
\usetikzlibrary{math,calc}

\usepackage{fullpage}

\usepackage{hyperref}

\numberwithin{equation}{section}

\theoremstyle{plain}
\newtheorem{thm}{Theorem}[section]

\newtheorem{lemma}[thm]{Lemma}
\newtheorem{prop}[thm]{Proposition}
\newtheorem{cor}[thm]{Corollary}

\theoremstyle{remark}
\newtheorem{rmk}[thm]{Remark}
\newtheorem{example}[thm]{Example}

\theoremstyle{definition}
\newtheorem{defn}[thm]{Definition}

\setlist{topsep=0.5em,itemsep=0.25em}

\newcommand{\R}{{\mathbb R}}

\newcommand{\N}{{\mathbb N}}

\newcommand{\bbS}{{\mathbb S}}

\DeclareMathOperator{\E}{\mathbb{E}}
\DeclareMathOperator{\PP}{\mathbb{P}}

\newcommand{\eps}{\epsilon}
\newcommand{\var}{\text{-}\mathrm{var}}

\newcommand{\SMALL}{\textstyle}

\newcommand{\mrd}{\mathop{}\!\mathrm{d}}
\newcommand{\id}{{\mathrm{id}}}

\newcommand{\cB}{{\mathcal B}}
\newcommand{\cF}{{\mathcal F}}
\newcommand{\cG}{{\mathcal G}}
\newcommand{\cJ}{{\mathcal J}}
\newcommand{\cM}{{\mathcal M}}
\newcommand{\cSJ}{{\mathcal{SJ}}}
\newcommand{\cSM}{{\mathcal{SM}}}

\newcommand{\cP}{{\mathcal P}}

\newcommand{\sD}{{\mathscr{D}}}
\newcommand{\bsD}{{\bar{\mathscr{D}}}}

\newcommand{\floor}[1]{\lfloor #1 \rfloor}

\DeclareMathOperator{\diam}{diam}

\DeclareMathOperator{\Lip}{Lip}
\DeclareMathOperator{\sgn}{sgn}

\newcommand{\tL}{{\tilde L}}
\newcommand{\tW}{{\widetilde W}}
\newcommand{\tX}{{\widetilde X}}
\newcommand{\tLambda}{{\tilde \Lambda}}

\newcommand{\balpha}{{\boldsymbol \alpha}}
\newcommand{\bsigma}{{\boldsymbol \sigma}}

\newcommand{\hX}{{\hat{X}}}

\DeclareFontFamily{U}{matha}{\hyphenchar\font45}
\DeclareFontShape{U}{matha}{m}{n}{
      <5> <6> <7> <8> <9> <10> gen * matha
      <10.95> matha10 <12> <14.4> <17.28> <20.74> <24.88> matha12
      }{}
\DeclareSymbolFont{matha}{U}{matha}{m}{n}
\DeclareFontSubstitution{U}{matha}{m}{n}
\DeclareFontFamily{U}{mathx}{\hyphenchar\font45}
\DeclareFontShape{U}{mathx}{m}{n}{
      <5> <6> <7> <8> <9> <10>
      <10.95> <12> <14.4> <17.28> <20.74> <24.88>
      mathx10
      }{}
\DeclareSymbolFont{mathx}{U}{mathx}{m}{n}
\DeclareFontSubstitution{U}{mathx}{m}{n}
\DeclareMathDelimiter{\vvvert}{0}{matha}{"7E}{mathx}{"17}

\title{Superdiffusive limits for deterministic fast-slow dynamical systems}

\author{
\hspace*{5em}
Ilya Chevyrev
\thanks{School of Mathematics,
University of Edinburgh,
Edinburgh EH9 3FD,
United Kingdom.
ichevyrev@gmail.com}
\and Peter K. Friz 
\thanks{Institut f\"ur Mathematik, Technische Universit\"at Berlin, and Weierstra\ss --Institut f\"ur Angewandte Analysis und Stochastik, Berlin, Germany.
friz@math.tu-berlin.de}
\hspace*{5em}
\and
Alexey Korepanov
\thanks{Department of Mathematics,
University of Exeter,
Exeter, EX4 4QF,
United Kingdom.
a.korepanov@exeter.ac.uk}
\and
Ian Melbourne
\thanks{Mathematics Institute,
University of Warwick,
Coventry, CV4 7AL,
United Kingdom.
i.melbourne@warwick.ac.uk}
}

\date{}

\begin{document}

 \maketitle

 \begin{abstract}
We consider deterministic fast-slow dynamical systems on $\mathbb{R}^m\times Y$ of the form
\[
  \begin{cases}
    x_{k+1}^{(n)} = x_k^{(n)} + n^{-1} a(x_k^{(n)}) + n^{-1/\alpha} b(x_k^{(n)}) v(y_k)\;, \\
    y_{k+1} = f(y_k)\;,
  \end{cases}
\]
where $\alpha\in(1,2)$. Under certain assumptions we prove convergence of the $m$-dimensional process $X_n(t)=
 x_{\lfloor nt \rfloor}^{(n)}$ to the solution of the stochastic differential equation
\[
    \mathop{}\!\mathrm{d} X = a(X)\mathop{}\!\mathrm{d} t + b(X) \diamond \mathop{}\!\mathrm{d} L_\alpha
   \; ,
  \]
where $L_\alpha$ is an $\alpha$-stable Lévy process and $\diamond$ indicates that the stochastic integral is in the Marcus sense. In addition, we show that our assumptions are satisfied for intermittent maps $f$ of Pomeau-Manneville type.
 \end{abstract}

\tableofcontents

\section{Introduction}
\label{sec:intro}

Averaging and homogenisation for systems with multiple timescales is
a longstanding and very active area of research~\cite{PavliotisStuart}.  
We focus particularly on homogenisation, where the limiting equation is a stochastic differential equation (SDE).
Recently there has been considerable interest in the case where the underlying multiscale system is deterministic, see~\cite{ CFKMZsub,Dolgopyat04,Dolgopyat05,
GM13,KM16,KM17,KKM20,MS11,Pene02} as well as our survey paper~\cite{CFKMZ}.
Almost all of this previous research has been concerned with the case where the limiting SDE is driven by Brownian motion.  
Here, we consider the case where the limiting SDE is driven by a superdiffusive $\alpha$-stable L\'evy process.

Let $\alpha\in(1,2)$.
The multiscale equations that we are interested in have the form
\begin{equation} \label{eq:fs}
  \begin{cases}
    x_{k+1}^{(n)} = x_k^{(n)} + n^{-1} a(x_k^{(n)}) + n^{-1/\alpha} b(x_k^{(n)}) v(y_k)\; , \\
    y_{k+1} = f(y_k)
  \end{cases}
\end{equation}
defined on $\R^m\times Y$ where $Y$ is a bounded metric space.
Here
\[
    a \colon \R^m\to\R^m
    \;, \quad
    b \colon \R^m\to\R^{m\times d}
    \;, \quad
    v \colon Y\to\R^d
    \;, \quad
    f\colon Y\to Y
    \;.
\]
It is assumed that the fast dynamical system $f\colon Y\to Y$ has an ergodic invariant probability measure $\mu$ and exhibits  superdiffusive behaviour; specific examples for such $f$ are described below.  
Let $v \colon Y\to\R^d$ be H\"older with $\int v\mrd\mu=0$.
Define for $n\ge1$,
\begin{equation} \label{eq:Wn}
    W_n(t)=n^{-1/\alpha}\sum_{j=0}^{\floor{nt}-1}v\circ f^j
    \;.
\end{equation}
Then $W_n$ belongs to $D([0,1], \R^d)$,
the Skorokhod space of c\`adl\`ag functions, and can be viewed
as a random process on the probability space $(Y, \mu)$ depending on the initial condition $y_0\in Y$.
As $n \to \infty$, the 
sequence of random variables $W_n(1)$ converges weakly in $\R^d$ to
an $\alpha$-stable law, and the process $W_n$ converges weakly in 
$D([0,1],\R^d)$ to the
corresponding $\alpha$-stable L\'evy process $L_\alpha$.

Now consider $x_0^{(n)}=\xi_n \in \R^m$, and solve~\eqref{eq:fs} to obtain
$(x_k^{(n)},y_k)_{k\ge0}$ depending on the initial condition $y_0\in(Y,\mu)$.  Define the c\`adl\`ag process $X_n\in D([0,1],\R^m)$ given by
$X_n(t) = x_{\floor{nt}}^{(n)}$;   again we view this as a process on $(Y,\mu)$.
Our aim is to show, 
under mild regularity assumptions on the functions $a\colon \R^m\to \R^m$ and $b \colon \R^m \to \R^{m \times d}$, that
$X_n\to_w X$ where $X$ is the solution of the SDE
\begin{equation} \label{eq:SDE}
    \mrd X = a(X)\mrd t + b(X) \diamond \mrd L_\alpha
    \; , \qquad X(0) = \xi
\end{equation}
and $\xi = \lim_{n\to\infty} \xi_n$.
Here, $\diamond$ indicates that the SDE is in the Marcus sense~\cite{Marcus80} (see~\cite{KPP95,Applebaum09,CP14} for the general theory of Marcus SDEs and their applications).

Previously such a result was shown
by Gottwald and Melbourne~\cite[Section~5]{GM13} in the special case $d=m=1$.  Generally the method in~\cite{GM13} works provided the noise is exact, that is $d=m$ and $b=(Dr)^{-1}$ for some diffeomorphism $r\colon \R^m\to\R^m$,
but cannot handle the general situation considered here where the noise term is typically not exact.  
There are three main complications:

\begin{itemize}
\item[(1)] In the case of exact noise, it is possible to reduce to the case
$b \equiv \id$ by a change of coordinates, similar to Wong-Zakai~\cite{WZ65}.
The general situation
necessitates the use of alternative tools such as rough paths.
In particular, weak convergence of $W_n$ is no longer sufficient and
we require in addition that $W_n$ is tight in $p$-variation.
This is shown in Theorem~\ref{thm:tight} below for specific examples, and in Section~\ref{sec:driver} for a large class of deterministic dynamical systems $f\colon Y\to Y$.

\item[(2)] Since the results for exact noise are achieved by a change of coordinates,
the sense of convergence for $W_n$ is inherited by $X_n$.
However, in general, even if $W_n\to_w L_\alpha$ in one of the standard
Skorokhod topologies~\cite{S56}, this need not be the case for $X_n$.
This phenomenon already appears in the simplest situations, as
illustrated in Example~\ref{ex:circle}.
Hence we have to consider convergence of $X_n$ in generalised Skorokhod topologies as introduced recently in Chevyrev and Friz~\cite{CF19}.

\item[(3)]  Rigorous results on convergence to $d$-dimensional stable L\'evy processes in  deterministic dynamical systems are only available for $d=1$, see~\cite{AaronsonDenker01,KPZ18,MZ15,TyranKaminska10}.  Hence one of the aims of this paper is to extend the dynamical systems theory to cover the case $d\ge2$.  
See Theorem~\ref{thm:PM} below for instances of this, and Section~\ref{sec:driver} for a general treatment.
\end{itemize}

In the remainder of the introduction, we discuss some of the issues associated to these three complications.  We also mention some examples of fast dynamical systems that lead to superdiffusive behaviour.
The archetypal such dynamical systems are the intermittent maps introduced by
Pomeau and Manneville~\cite{PomeauManneville80}.
Perhaps the simplest example~\cite{LSV99} is the map $f\colon Y\to Y$, $Y=[0,1]$, with a neutral fixed point at $0$:
\begin{equation} \label{eq:LSV}
    f(y) =
    \begin{cases}
        y (1 + 2^{1/\alpha} y^{1/\alpha})\;, & y \in [0, \frac12) \; , \\
        2y - 1\;, & y \in [\frac12,1] \; .
    \end{cases}
\end{equation}
See Figure~\ref{fig:PM}(a).
Here, $\alpha>0$ is a real parameter and there is a unique absolutely continuous invariant probability measure $\mu$ for $\alpha>1$.
Let $v\colon Y\to\R$ be H\"older with $\int_Y v\mrd\mu=0$ and $v(0)\neq0$, and
define $W_n$ as in~\eqref{eq:Wn}.
For $\alpha\in(1,2)$ it was shown by~\cite{Gouezel04} (see also~\cite{Zweimuller03}) that $W_n(1)$ converges in distribution to an $\alpha$-stable law.  By~\cite{MZ15},
the process $W_n$ converges weakly to the corresponding L\'evy process $L_\alpha$ in the $\cM_1$ Skorokhod topology on $D([0,1],\R)$.  

Now let $d\ge2$.  There are two versions of the $\cM_1$ topology on
$D([0,1],\R^d)$, see~\cite[Chapter~3.3]{Whitt02}.
In this paper we use the strong topology $\cSM_1$.
For $v\colon Y\to\R^d$ H\"older with $\int_Y v\mrd \mu=0$ and $v(0)\neq0$, we prove
convergence of $W_n$ to a $d$-dimensional L\'evy process $L_\alpha$ in the $\cSM_1$ topology.

The example~\eqref{eq:LSV} is somewhat oversimplified for our purposes since
$L_\alpha$ is essentially one-dimensional, being supported on the line $\{c v(0): c \in \R\}$.  This structure can be exploited in proving that $W_n\to_w L_\alpha$,
though it is not clear if this simplifies the homogenisation result $X_n\to_w X$.
To illustrate that we do not rely on one-dimensionality of the limiting process in any way, we consider an example with two neutral fixed points.
(It is straightforward to extend to maps with a larger number of neutral fixed points.)
Accordingly, our main example is the intermittent map
$f \colon Y \to Y$, $Y= [0,1]$, with two
symmetric neutral fixed points at $0$ and~$1$:
\begin{equation}
    \label{eq:PM}
    f(y) = \begin{cases} 
        y(1+3^{1/\alpha} y^{1/\alpha})\;, & y \in[0, \frac{1}{3}) \; , \\
        3y-1\;, & y\in [\frac{1}{3} , \frac{2}{3}) \; , \\
        1 - (1-y) (1 + 3^{1/\alpha} (1-y)^{1/\alpha})\;, & y\in[ \frac{2}{3},1] \; .
  \end{cases}
\end{equation}
See Figure~\ref{fig:PM}(b).
Again $\alpha>0$ is a real parameter, there is a unique absolutely continuous invariant probability measure $\mu$ for $\alpha>1$, and we restrict to the range 
$\alpha\in(1,2)$.

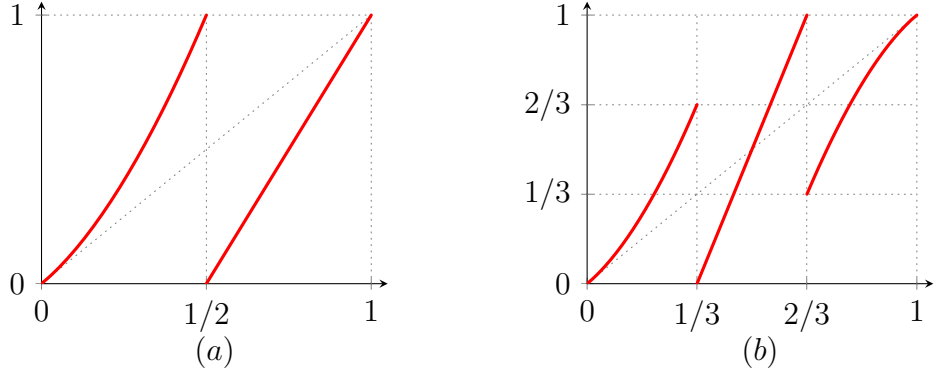
\begin{figure}
\[
     \begin{tikzpicture}
         \tikzmath{ \ainv = 0.9; } 
     \begin{axis}[
       name =Wn,
       height       = 2.1in,
       xmax         = 1.05,
       ymax         = 1.05,
       xtick        = {0.0, 0.5, 1.0},
       xticklabels  = {0, $1/2$, 1},
       axis x line  = bottom,
       ytick        = {0.0, 1.0},
       yticklabels  = {0, 1},
       axis y line  = left,
       line cap=round
     ]
         \addplot[gray,dotted] coordinates { (0,1) (1,1) };
         \addplot[gray,dotted] coordinates { (1,0) (1,1) };
         \addplot[gray,dotted] coordinates { (0,0) (1,1) };
         \addplot[gray,dotted] coordinates { (0.5,0) (0.5,1) };
         \addplot[red,very thick,domain=0:0.5, samples=20]{ x*(1+(2*x)^(\ainv)) };
         \addplot[red,very thick,domain=0.5:1.0, samples=2]{ 2*x - 1 };
     \end{axis}
     \node[anchor=north] at (Wn.south) {$\begin{matrix} \\ (a) \end{matrix}$};
     \end{tikzpicture}
     \qquad \qquad
     \begin{tikzpicture}
         \tikzmath{ \ainv = 0.9; } 
     \begin{axis}[
       name =Wn,
       height       = 2.1in,
       xmax         = 1.05,
       ymax         = 1.05,
       xtick        = {0.0, 0.333, 0.666, 1.0},
       xticklabels  = {0, $1/3$, $2/3$, 1},
       axis x line  = bottom,
       ytick        = {0.0, 0.333, 0.666, 1.0},
       yticklabels  = {0, $1/3$, $2/3$, 1},
       axis y line  = left,
       line cap=round
     ]
         \addplot[gray,dotted] coordinates { (0,1) (1,1) };
         \addplot[gray,dotted] coordinates { (1,0) (1,1) };
         \addplot[gray,dotted] coordinates { (0,0) (1,1) };
         \addplot[gray,dotted] coordinates { (0,0.333) (1,0.333) };
         \addplot[gray,dotted] coordinates { (0,0.666) (1,0.666) };
         \addplot[gray,dotted] coordinates { (0.333,0) (0.333,1) };
         \addplot[gray,dotted] coordinates { (0.666,0) (0.666,1) };
         \addplot[red,very thick,domain=0:(1/3), samples=20]{ x*(1+(3*x)^(\ainv)) };
         \addplot[red,very thick,domain=(1/3):(2/3), samples=2]{ 3*x - 1 };
         \addplot[red,very thick,domain=(2/3):1, samples=20]{ 1 - (1-x)*(1+(3*(1-x))^(\ainv)) };
     \end{axis}
     \node[anchor=north] at (Wn.south) {$\begin{matrix} \\ (b) \end{matrix}$};
     \end{tikzpicture}
 \]

\vspace{-2ex}
\caption{Examples of intermittent maps:  (a) the map~\eqref{eq:LSV}, (b) the map~\eqref{eq:PM}.}
\label{fig:PM}
\end{figure}

As part of a result for a general class of nonuniformly expanding maps (Section~\ref{sec:driver}) we prove:
\begin{thm}
    \label{thm:PM}
    Consider the intermittent map~\eqref{eq:LSV} or~\eqref{eq:PM} with $\alpha\in(1,2)$
    and let $v\colon Y\to\R^d$ be H\"older with $\int_Y v\mrd\mu=0$ and $v(0)\neq0$,
    also $v(1)\neq0$ in case of~\eqref{eq:PM}.
Let $\PP$ be any probability measure on $Y$ that is absolutely continuous with respect to Lebesgue, and regard $W_n$ as a process on $(Y,\PP)$.
    Then 
\[
W_n\to_w L_\alpha\;\text{in $D([0,1],\R^d)$ with the
    $\cSM_1$  topology as $n\to\infty$}\;,
\]
    where $L_\alpha$ is a $d$-dimensional $\alpha$-stable L\'evy process.
\end{thm}

\begin{rmk}
    \label{rmk:PM}
    The limiting process $L_\alpha$ is explicitly identified in Subsection~\ref{sec:PM}.
\end{rmk}

In the context of~\cite{GM13}, the conclusion $W_n\to_w L_\alpha$
was sufficient to prove the homogenisation result $X_n\to_w X$.
This is not the case for general noise, and we require tightness in $p$-variation.
For $1 \le p<\infty$, recall that
the \emph{$p$-variation} of $u\colon [0,1]\to\R^d$ is given by
\begin{equation} \label{eq:pvar}
  \|u\|_{p\var}
  = \sup_{0=t_0<t_1<\dots<t_k=1}
  \Bigl(\sum_{j=1}^k\bigl|u(t_j) - u(t_{j-1})\bigr|^p\Bigr)^{1/p}
  \;,
\end{equation}
where $| \cdot | $ denotes the Euclidean norm on $\R^d$.

\begin{thm} \label{thm:tight}
Consider the intermittent map~\eqref{eq:LSV} or~\eqref{eq:PM} with $\alpha\in(1,2)$
and let $v\colon Y\to\R^d$ be H\"older with $\int_Y v\mrd\mu=0$.
Let $\PP$ be any probability measure on $Y$ that is absolutely continuous with respect to Lebesgue.
Then the family of random variables $\|W_n\|_{p\var}$ is tight on $(Y,\PP)$ for all $p>\alpha$.
\end{thm}

The main abstract result in this paper states that
the properties established in Theorems~\ref{thm:PM} and~\ref{thm:tight}
are the key ingredients required to solve the homogenisation problem.  Informally:
\begin{quote}
    Consider the fast-slow system~\eqref{eq:fs} and define
    $W_n$ as in~\eqref{eq:Wn} and $X_n=x_{\floor{nt}}^{(n)}$ with $x^{(n)}_0=\xi_n$.
    Suppose that $\lim_{n\to\infty}\xi_n = \xi$, $W_n\to_w L_\alpha$, an $\alpha$-stable L\'evy process,
    in $D([0,1],\R^d)$ with the $\cSM_1$-topology, and
    that $\|W_n\|_{p\var}$ is tight for all $p>\alpha$.

    If $v$ is bounded and $a,\,b$ are sufficiently smooth, then 
    $X_n\to_w X$ in $D([0,1],\R^m)$ where $X$ is the solution to the SDE~\eqref{eq:SDE}.
\end{quote}

We give a rigorous formulation of this result in Theorem~\ref{thm:rig} (in the above statement we assume that the limiting process is L{\'e}vy only for convenience -- the result holds true for an arbitrary limiting process as seen from Theorem~\ref{thm:rig}).
To complete the statement, it is necessary to describe the topology on $D([0,1],\R^m)$ in which $X_n$ converges.
As already indicated, the $\cSM_1$ topology is too strong in general.
The next example illustrates where the problem lies.


\begin{example}
  \label{ex:circle}
  Let $\theta > 0$ and consider continuous deterministic processes $W_n \colon [0,1] \to \R$ which are equal to $0$ on $[0,\frac12]$,
  equal to $\theta$ on $[\frac12 + \frac{1}{n}, 1]$, and linear on $[\frac12, \frac12 + \frac{1}{n}]$.
  Let $X_n = (X_n^1, X_n^2)$ be the solution of the ordinary differential equation
  \[
    \begin{pmatrix}
      \mrd X_n^1 \\ \mrd X_n^2
    \end{pmatrix}
    =
    \begin{pmatrix}
      - X_n^2 \\ X_n^1
    \end{pmatrix}
    \mrd W_n
    \;, \qquad
    \begin{pmatrix}
      X_n^1(0) \\ X_n^2(0)
    \end{pmatrix}
    =
    \begin{pmatrix}
      1 \\ 0
    \end{pmatrix}
    \;.
  \]
The graphs of $W_n$ and $X_n$ are shown in Figure~\ref{fig:circle}.

  It is easy to see that $W_n$ converges to $\theta \, 1_{[1/2,1]}$
  in the $\cM_1$ topology as $n \to \infty$, and that
  $(X_n^1, X_n^2) = (\cos W_n , \sin W_n)$.
  The process $X_n$ converges pointwise to 
  \[
      X(t) = \begin{cases}
          (1,0)\;, &  t \leq \frac12\;, \\
          (\cos\theta, \sin\theta)\;, & t > \frac12 \; .
      \end{cases}
  \]
  In particular, if $\theta = 2 \pi$, then $X \equiv (1,0)$ is continuous.
  At the same time, $X_n$ fails to converge in any
  of the Skorokhod topologies.
\end{example}

\begin{figure}[ht]
  \[
      \begin{tikzpicture}
          \tikzmath{ \hhh = 6.5; \hhhh = \hhh * 1.1; }
          \begin{axis}[
              name =Wn,
              height       = 1.9in,
              xmax         = 1.05,
              ymax         = \hhhh ,
              xtick        = {0.0, 0.5, 0.6, 1.0},
              xticklabels  = {0, $\frac{1}{2}$, , 1},
              ytick        = {0.0, \hhh },
              yticklabels  = {0, $\theta$},
              axis lines   = center,
              line cap=round
              ]
              \addplot[color=gray,dotted] coordinates { (0,\hhh) (0.6,\hhh) };
              \addplot[color=red,very thick] coordinates { (0,0) (0.5,0) };
              \addplot[color=red,very thick] coordinates { (0.5,0) (0.6, \hhh ) };
              \addplot[color=red,very thick] coordinates { (0.6, \hhh ) (1, \hhh ) };
          \end{axis}
          \node[anchor=north] at (Wn.south) {$\begin{matrix} \\ \\ W_n \end{matrix}$};
      \end{tikzpicture}
      \qquad
      \begin{tikzpicture}
          \tikzmath{ \hhh = 6.5; }
          \begin{axis}[
              name       = Zn,
              line cap   = round,
              view       = {-25}{-25},
              axis lines = center,
              xmax       = 1.1,
              ymax       = 1.1,
              zmax       = 1.1,
              height     = 2.4in,
              xtick      = {1},
              xticklabel = {$X_n^1$},
              xtick style={draw=none},
              ytick      = {1},
              yticklabel = {$X_n^2$},
              ytick style={draw=none},
              ztick      = {1},
              zticklabel = {$t$},
              ztick style={draw=none},
              ]
              \addplot3+ [
                  domain     = 0:\hhh ,
                  samples    = 64,
                  samples y  = 0,
                  mark       = none,
                  very thick,
                  color      = red
                  ] ( {cos(deg(x))},{sin(deg(x)},{0.5 + 0.1 * x / \hhh } );
              \addplot3+ [
                  color      = red,
                  mark       = none,
                  very thick
                  ] coordinates { (1,0,0) (1,0,0.5)};
              \addplot3+ [
                  color      = red,
                  mark       = none,
                  very thick
                  ] coordinates { ({cos(deg(\hhh))},{sin(deg(\hhh))},0.6)
                  ({cos(deg(\hhh))},{sin(deg(\hhh))},1.0) };
          \end{axis}
          \node[anchor=north] at (Zn.south) {$\begin{matrix} \\ \\ X_n \end{matrix}$};
      \end{tikzpicture}
      \qquad
      \begin{tikzpicture}
          \tikzmath{ \hhh = 6.5; \hhhh = \hhh * 1.1; }
          \begin{axis}[
              name =Znc,
              height       = 1.9in,
              xmax         = 1.05,
              ymin         = -1.2,
              ymax         = 1.2,
              xtick        = {0.0, 0.5, 0.6, 1.0},
              xticklabels  = {0, $\frac{1}{2}$, , 1},
              ytick        = {-1, 0, 1},
              yticklabels  = {-1, 0, 1},
              axis lines   = center,
              line cap=round
              ]
              %
              %
              \addplot[color=red,very thick,dotted] coordinates {(0,0) (0.5, 0)};
              \addplot[color=red,very thick,dotted,domain=0.5:0.6] {sin(deg((x-0.5) / 0.1 * \hhh))};
              \addplot[color=red,very thick,dotted] coordinates {(0.6, {sin(deg(\hhh))}) (1.0, {sin(deg(\hhh))})};
          \end{axis}
          \node[anchor=north] at (Znc.south) {$\begin{matrix} \\ \\ X_n^2 \end{matrix}$};
      \end{tikzpicture}
  \]
\caption{Graphs of $W_n$ and $X_n=(X_n^1,X_n^2)$ in Example~\ref{ex:circle}.}
\label{fig:circle}
\end{figure}
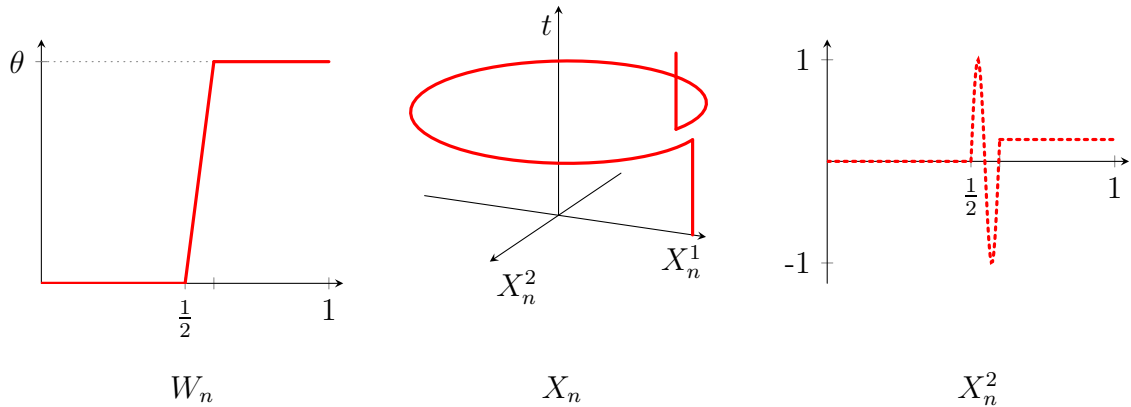

The problem outlined in Example~\ref{ex:circle} arises naturally in the fast-slow system~\eqref{eq:fs}.
Figure~\ref{fig:plots}
illustrates a realisation\footnote{Generated from~\url{https://khu.dedyn.io/work/scaled-graphs/fast-slow/}}
of $W_n$ and $X_n$ for $d=m=2$ and the map~\eqref{eq:PM}.
The function $b$ is taken as
\[
    b(x_1,x_2)\binom{v_1}{v_2} = \binom{-x_2}{x_1}v_1 + \binom{x_1}{x_2}v_2\;.
\]
Note that, although $W_n$ appears to converge in $\cSM_1$ in accordance with Theorem~\ref{thm:PM}, $X_n$ moves along the integral curves of a vector field, and thus does not approximate its limit in $\cSM_1$. 

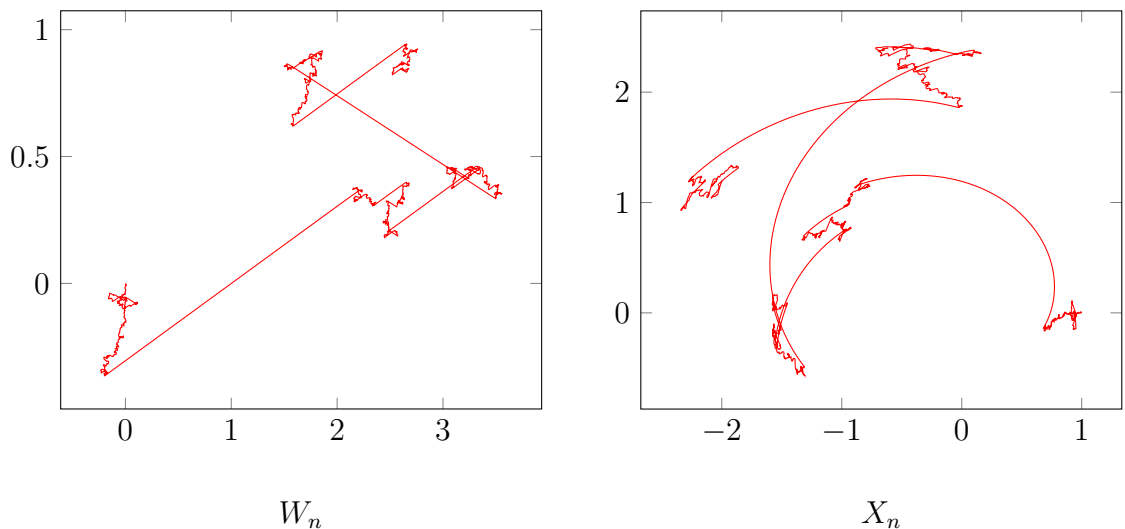
\begin{figure}[ht]
    \[
        \begin{tikzpicture}
            \begin{axis}[
                name=W,
                width = 3.14in,
                xticklabel style={/pgf/number format/fixed},
                yticklabel style={/pgf/number format/fixed}
                ]
                \addplot[red] table [x=Wx, y=Wy] {fs.data};
            \end{axis} 
            \node[anchor=north] at (W.south) {$\begin{matrix} \\ \\ W_n \end{matrix}$};
        \end{tikzpicture}
        \qquad
        \begin{tikzpicture}
            \begin{axis}[
                name=X,
                width = 3.14in,
                xticklabel style={/pgf/number format/fixed},
                yticklabel style={/pgf/number format/fixed}
                ]
                \addplot[red] table [x=Xx, y=Xy] {fs.data};
            \end{axis} 
            \node[anchor=north] at (X.south) {$\begin{matrix} \\ \\ X_n \end{matrix}$};
        \end{tikzpicture}
    \]
    \caption{Realisation of $W_n$ and $X_n$ with $n=10^4$ points}
    \label{fig:plots}
\end{figure}

Topologies naturally suited for convergence in
Example~\ref{ex:circle} were recently introduced in~\cite{CF19}.
These topologies are a generalisation of the Skorokhod $\cSM_1$ topology which allow for convenient control of differential equations.
Briefly, jumps of a c\`adl\`ag process are interpreted as an
instant travel along prescribed continuous paths which depend only on
the start and end points of the jump.
The full ``pathspace'' thus becomes the set of pairs $(X,\phi)$, where $X\colon[0,1]\to\R^d$ is a c{\`a}dl{\`a}g path and $\phi$ is a so-called \textit{path function}~\cite{Chevyrev18} which maps each jump $(X(t-),X(t))$ to a continuous path from $X(t-)$ to $X(t)$.
It is often convenient to fix $\phi$, which in turn determines a topology on c{\`a}dl{\`a}g paths; if $\phi$ is linear, one recovers the
$\cSM_1$ topology.
For our purposes, it is necessary to adapt the spaces introduced in~\cite{CF19},
and we give details in Sections~\ref{sec:prep} and~\ref{sec:rp}.

The paper is organised as follows. In Section~\ref{sec:prep}, we introduce the necessary prerequisites on generalised Skorokhod topologies and Marcus differential equations in order to state rigorously our main abstract result
Theorem~\ref{thm:rig}.
The proof is given at the end of Section~\ref{sec:rp} after introducing the necessary results from rough path theory.
In Sections~\ref{sec:GM} to~\ref{sec:driver},
we show that a class of nonuniformly expanding dynamical systems, including~\eqref{eq:LSV} and~\eqref{eq:PM}, satisfies the conclusions of
Theorems~\ref{thm:PM} and~\ref{thm:tight} which are in turn the main hypotheses of Theorem~\ref{thm:rig}.
Section~\ref{sec:GM} deals with a class of uniformly expanding maps known as Gibbs-Markov maps, and Section~\ref{sec:induce} provides the inducing step to pass from uniformly expanding maps to nonuniformly expanding maps.
In Section~\ref{sec:driver}, we apply the results of Sections~\ref{sec:GM} and~\ref{sec:induce} to
the intermittent maps~\eqref{eq:LSV} and~\eqref{eq:PM}.
The precise result on homogenisation of the system~\eqref{eq:fs} with fast dynamics given by either~\eqref{eq:LSV} or~\eqref{eq:PM} is stated in Corollary~\ref{cor:PM}.

\vspace{-2ex}
\paragraph{Notation}
We use ``big O'' and $\lesssim$ notation interchangeably, writing $a_n=O(b_n)$ or $a_n\lesssim b_n$
if there is a constant $C>0$ such that
$a_n\le Cb_n$ for all sufficiently large $n$.
As usual, $a_n=o(b_n)$ means that $\lim_{n\to\infty}a_n/b_n=0$
and $a_n\sim b_n$ means that $\lim_{n\to\infty}a_n/b_n=1$.

\vspace{-2ex}
\paragraph{Acknowledgements}
I.C. was funded by a Junior Research Fellowship of St John's College, Oxford while this work was carried out.
P.K.F. acknowledges partial support from the ERC, CoG-683164, the Einstein Foundation Berlin, and DFG research unit FOR2402.
A.K. and I.M. acknowledge partial support from the European Advanced Grant StochExtHomog (ERC AdG 320977).
A.K. is also supported by an Engineering and Physical Sciences Research Council grant
EP/P034489/1.
We would like to thank the anonymous referees for their helpful and detailed comments.

\section{Setup and result}
\label{sec:prep}

In this section, we collect the material necessary to formulate our main abstract result Theorem~\ref{thm:rig}.

\subsection{Skorokhod topologies}

Let $D = D([0,1],\R^d)$ denote the Skorokhod space of c{\`a}dl{\`a}g functions,
i.e.\ the set of functions $X \colon [0,1] \to \R^d$ which are right-continuous with left limits.
For $X\in D$ and $t\in [0,1]$, we denote $X(t-) = \lim_{s \nearrow t} X(s)$, with the convention that $X(0-) = X(0)$.

Let $\Lambda$ denote the set of all increasing bijections
$\lambda \colon [0,1] \to [0,1]$ and let $\id\in\Lambda$ denote the identity map $\id(t)=t$.
For $X_1, X_2 \in D$, let
$\bsigma_\infty(X_1, X_2)$ be the Skorokhod distance
\[
  \bsigma_\infty(X_1, X_2)
  = \inf_{\lambda\in \Lambda} \max \{
    \|\lambda - \id\|_\infty, \|X_1 \circ \lambda - X_2\|_\infty 
  \}\;,
\]
where $\|X\|_\infty = \sup_{t\in[0,1]}|X(t)|$.
The topology on $D$ induced by $\bsigma_\infty$ is known as the strong $\cJ_1$, or $\cSJ_1$, topology.

Another important topology on $D$ is the strong $\cM_1$, or $\cSM_1$, topology
defined as follows.
For $X\in D$ consider the ``completed'' graph
$\Gamma(X) = \{(t,x) \in [0,1]\times \R^d : x\in [X(t-),X(t)]\}$,
and let $\Lambda^*(X)$ be the set of all continuous bijections
$(\lambda,\gamma) \colon [0,1] \to \Gamma(X)$ with $\lambda(0)=0$.
Then the $\cSM_1$ topology on $D$ is induced by the metric
\[
    d_{\cSM_1}(X_1, X_2)
    = \inf_{\substack{(\lambda_i,\gamma_i)\in\Lambda^*(X_i)\\i=1,2}}
    \max\{\|\lambda_1-\lambda_2\|_\infty, \|\gamma_1- \gamma_2\|_\infty \}
    \; .
\]

 \subsection{Generalised \texorpdfstring{$\cSM_1$}{SM1} topologies}
 \label{subsec:generalised_SM1}
 
We now introduce generalisations of the $\cSM_1$ topology from~\cite{CF19}.

For $1 \le p<\infty$, recall the $p$-variation $\|u\|_{p\var}$ of $u\colon [0,1]\to\R^d$ defined by~\eqref{eq:pvar}.
We furthermore denote
$\vvvert u \vvvert_{p\var} = |u(0)| + \|u\|_{p\var}$.
Let
\[
D^{p\var} = \{u\in D([0,1],\R^{d}) : \|u\|_{p\var}<\infty\}
\]
and $C^{p\var}([0,1],\R^d)\subset D^{p\var}$ be the set of $u\in D^{p\var}$ which are continuous.
Let $\bsigma_{p\var}$ denote the Skorokhod-type $p$-variation on $D^{p\var}$:
\[
  \bsigma_{p\var}(X_1, X_2) 
  = \inf_{\lambda\in\Lambda} \max \{
    \|\lambda - \id\|_\infty, \vvvert X_1 \circ \lambda - X_2 \vvvert_{p\var} 
  \}\;.
\]

\begin{defn}\label{def:path_function}
A \emph{path function} on $\R^d$ is a map $\phi \colon J \to C([0,1], \R^d)$, where $J \subset \R^d\times \R^d$, for which $\phi(x,y)(0) = x$ and $\phi(x,y)(1) = y$ for all $(x,y)\in J$.
For a path $X \in D([0,1], \R^d)$, we say that $t\in[0,1]$ is a jump time of $X$ if $X(t-) \neq X(t)$.
  A pair $(X,\phi)$ is called admissible if all the jumps of $X$ are in the domain of definition of $\phi$,
  i.e.\ $(X(t-),X(t)) \in J$ for all jump times $t$ of $X$.
  We denote by $\bsD([0,1],\R^d)$ the space of admissible pairs $(X, \phi)$.
  We let $\sD([0,1],\R^d) = \bsD([0,1],\R^d) / \sim$, where $(X_1, \phi_1) \sim (X_2, \phi_2)$
  if $X_1 = X_2$ and $\phi_1(X_1(t-),X_1(t))$ is a reparametrisation of $\phi_2(X_1(t-),X_1(t))$ for all jump times $t$ of $X_1$.
\end{defn}

\begin{rmk}
    We often keep implicit the interval $[0,1]$ and $\R^d$, as well as $J$, when they are clear from the context.
    We allow $J$ to be a strict subset of $\R^d\times \R^d$ since this case arises naturally when considering driver-solution pairs for canonical differential equations, see the final discussion in Section~\ref{subsec:Marcus}. 
\end{rmk}

A simple path function which shall play an important role is the following.

\begin{defn}
The \emph{linear path function} on $\R^{k}$ is the map $\ell_k\colon\R^k\times\R^k \to C([0,1],\R^k)$
defined by $\ell_k(x,y)(t) = x+t(y-x)$ for all $x,y\in \R^{k}$.
\end{defn}

Fix a sequence $r_1, r_2, \ldots > 0$ with $\sum_j r_j < \infty$.
Given $(X, \phi) \in \bsD$ and 
$\delta > 0$, let $X^{\phi, \delta} \in C([0,1], \R^d)$ denote the
continuous version of $X$, where the $k$-th largest jump is made
continuous using $\phi$ on a fictitious time interval of length
$\delta r_k$. More precisely:
\begin{itemize}
    \item Let $m\ge0$ be the number of jumps (possibly infinite) of $X$.
        We order the jump times $\{t_j\}_{j=1}^m$ 
        so that $|X(t_k) - X(t_k-)| \geq |X(t_{k+1}) - X(t_{k+1}-)|$
        for each $k$, with $t_k < t_{k+1}$ in case of equality.
    \item Let $r= \sum_{j=1}^m r_j$ and define the map
        \begin{equation}
            \label{eq:tau}
            \tau \colon [0,1] \to [0, 1+\delta r]\;, \quad 
            \tau(t) = t + \sum_k \delta r_k 1_{\{t_k \leq t\}}\;.
        \end{equation}
    \item Define an intermediate process $\hX \in C([0,1+\delta r], \R^d)$,
        \[
            \hX (t) = \begin{cases}
                X(s) & \text{if } t = \tau(s) \text{ for some } s \in [0,1]\;, \\
                \phi(X(t_k-), X(t_k))\bigl(\frac{s - \tau(t_k-)}{\delta r_k}\bigr)
                     & \text{if } t \in [\tau(t_k-), \tau(t_k)) \text{ for some } k\;.
            \end{cases}
        \]
    \item Finally, let $X^{\phi, \delta}(t) = \hX(t (1+\delta r))$,
        scaling the domain of $\hX$ from $[0, 1+\delta r]$ to $[0,1]$. 
\end{itemize}

For $(X, \phi) \in \sD([0,1],\R^{d})$ and $p \geq 1$, let
\[
  \|(X, \phi)\|_{p\var} = \|X^{\phi, 1}\|_{p\var}
  \;.
\]
Note that $\|(X, \phi)\|_{p\var}$ is well-defined since $\|X^{\phi, 1}\|_{p\var}$ depends on
neither the parametrisation of $\phi$, nor the sequence $\{r_k\}$. 
Let
\[
  \sD^{p\var} = \{
    (X, \phi) \in \sD : \|(X, \phi)\|_{p\var} < \infty
  \}
  \;.
\]
Given $(X_1, \phi_1)$ and $(X_2, \phi_2)$ in $\sD^{p\var}$, let
\[
  \balpha_{p\var} ((X_1, \phi_1), (X_2, \phi_2))
  = \lim_{\delta \to 0} \bsigma_{p\var} (X_1^{\phi_1, \delta}, X_2^{\phi_2, \delta})\;,
\]
which defines a metric on $\sD^{p\var}$~\cite[Remark~3.8]{CF19}.

\subsection{Marcus differential equations}
\label{subsec:Marcus}

For $\gamma>0$, let $C^\gamma(\R^m,\R^n)$ denote the space of functions $b \colon \R^m \to \R^n$ such that
\[
    \|b\|_{C^\gamma}
    = \max_{|\alpha|=0,\ldots,\floor{\gamma}} \|D^\alpha b\|_\infty
    + \sup_{x,y\in \R^m} \max_{|\alpha|=\floor\gamma}
    \frac{|D^\alpha b(x) - D^\alpha b(y)|}{|x-y|^{\gamma-\floor\gamma}}
    < \infty\;.
\]
Note that our notation is slightly non-standard since $b\in C^N$ for $N\in \N$ implies only that the $(N-1)$-th derivative of $b$ is Lipschitz rather than continuous.

Suppose that $W \in D^{p\var}([0,1], \R^{d})$ with $1\leq p < 2$,
and that $a\in C^\beta(\R^m,\R^m)$ and $b \in C^{\gamma}(\R^m,\R^{m\times d})$ 
with $\beta > 1$ and $\gamma > p$.
Under these conditions, we can define and solve (in a purely deterministic way)
a Marcus-type differential equation
\begin{equation}\label{eq:Marcus_W}
    \mrd X = a(X) \mrd t + b(X) \diamond \mrd W
    \; .
\end{equation}
The solution is obtained as follows from the theory of
continuous rough differential equations (RDEs) in the Young regime~\cite{Lyons94,FV10,FH14}.
Consider the c{\`a}dl{\`a}g path $\tW \colon [0,1]\to \R^{1+d}$ given by $\tW(t) = (t,W(t))$.
Using the notation of Section~\ref{subsec:generalised_SM1}, consider the continuous path
$\tW^{\phi,1} \colon [0,1+r] \to \R^d$, where $\phi=\ell_{1+d}$ is the linear path function
on $\R^{1+d}$.
Let $\tau \colon [0,1] \to [0,1+r]$ be the corresponding map given by~\eqref{eq:tau}.
Then $\|\tW^{\phi,1}\|_{p\var}\lesssim \|W\|_{p\var}$ (see e.g.~\cite[Corollary~A.6]{Chevyrev18}), and therefore one can solve the (continuous) RDE
\[
\mrd \tX = (a,b)(\tX) \mrd \tW\;.
\]
The solution is a continuous path $\tX \colon [0,1+r] \to \R^m$ of finite $p$-variation.
The solution to~\eqref{eq:Marcus_W} is the c{\`a}dl{\`a}g path $X \colon[0,1]\to \R^d$
given by $X(t) = \tX(\tau(t))$.
We discuss a more general interpretation of this equation in Section~\ref{subsec:rough_int}.

\begin{rmk}
In the case that $W$ is a semimartingale, one can verify that $X$ is the solution to the classical Marcus SDE (see~\cite[Proposition~4.16]{CF19} for the general case $p>2$ but with stronger regularity assumptions on $a,b$; the proof carries over to our setting without change).
\end{rmk}

To properly describe solutions of~\eqref{eq:Marcus_W} and regularity
of the solution map $W \mapsto X$,
it is not enough to look at $X$ as an element of $D([0,1],\R^m)$.
As in Example~\ref{ex:circle}, one may have $X \equiv 0$ say, but
with sizeable jumps in fictitious time.

Following~\cite{CF19}, we consider the
driver-solution space $D([0,1], \R^{d + m})$, made to contain the pairs $(W, X)$, and introduce a new path function on $\R^{d+m}$.

\begin{defn}\label{def:phi_b_def}
Consider $b\in C^{1}(\R^m,\R^{m\times d})$.
For $x \in \R^m$ and $\Phi \in C^{1\var}([0,1],\R^{d})$,
let $\pi_b[x;\Phi] \in C^{1\var}([0,1], \R^m)$ denote the solution $\Pi$ of the equation
\begin{equation*}
    \mrd \Pi = b(\Pi) \mrd \Phi
    \;, \quad
    \Pi(0) = x
    \;.
\end{equation*}
We define the path function $\phi_b$ on $\R^{d+m}$ by
\begin{equation}
    \phi_b\bigl((w_1,x_1), (w_2,x_2) \bigr)(t) = \bigl(\ell_d(w_1,w_2)(t), \pi_b[x_1;\ell_{d}(w_1,w_2)](t)\bigr)
    \; ,
\end{equation}
which is defined on
\[
    J_b = \bigl\{\bigl((w_1,x_1), (w_2,x_2) \bigr)
    : w_1,w_2\in\R^{d}\;,\; \pi_b[x_1;\ell_{d}(w_1,w_2)](1) = x_2 \bigr\}
    \;.
\]
\end{defn}

Note that $J_b$ is a strict subset of $\R^{d+m}\times \R^{d+m}$.
Observe that if $X$ solves~\eqref{eq:Marcus_W}, then $((W,X),\phi_b)\in \sD^{p\var}([0,1],\R^{d+m})$ and the path function $\phi_b$ describes how the discontinuities of $(W,X)$
are traversed in fictitious time.

\subsection{Main abstract result}
\label{sec:rig}

Now we are ready for a rigorous formulation of the main abstract result.
Consider the fast-slow system~\eqref{eq:fs}
with initial condition $x^{(n)}_0 = \xi_n$ such that $\lim_{n\to\infty}\xi_n = \xi$.
Suppose that $\alpha\in(1,2)$,
$\alpha'\in[\alpha,2)$,
$v\in L^\infty(Y,\R^d)$,
$a\in C^{\beta}(\R^m,\R^m)$, $b \in C^\gamma(\R^m,\R^{m\times d})$
for some $\beta>1$, $\gamma > \alpha'$.
Define $W_n$ as in~\eqref{eq:Wn} and $X_n(t) = x_{\floor{nt}}^{(n)}$.

\begin{thm}
    \label{thm:rig}
Suppose that  
\begin{itemize}
\item $W_n\to_w L$ 
in $D([0,1],\R^d)$ with the
$\cSM_1$  topology as $n\to\infty$ for some process $L.$
\item $\|W_n\|_{p\var}$ is tight for all $p>\alpha'$.
\end{itemize}
Then, for all $p>\alpha'$, it holds that $\|L\|_{p\var}<\infty$ a.s.\ and
    \[
        ((W_n, X_n), \ell_{d+m}) \to_w ((L, X), \phi_{b})
        \qquad \text{as} \qquad n \to \infty
    \]
in $(\sD^{p\var}([0,1],\R^{d +m}), \balpha_{p\var})$, where $X$ is the solution of the Marcus differential equation
    \begin{equation}\label{eq:Marcus}
    \mrd X = a(X) \mrd t + b(X) \diamond \mrd L
        \;, \qquad X(0) = \xi \in \R^m
        \;.
    \end{equation}
\end{thm}

The proof of Theorem~\ref{thm:rig} is given at the end of Section~\ref{sec:rp}.

\begin{rmk}
\begin{enumerate}[label=(\alph*)]
\item The property $\|L\|_{p\var}<\infty$ a.s.\ together with $\gamma>\alpha'$ guarantees that the Marcus equation~\eqref{eq:Marcus} admits a unique solution for a.e.\ realisation of $L$.
In our applications, 
$L$ is an $\alpha$-stable L{\'e}vy process, for which the finiteness of $\|L\|_{p\var}$ is classical,
and we take $\alpha'=\alpha$.
We introduce the parameter $\alpha'$ to highlight that the threshold for the value of $p$ in the second condition of Theorem~\ref{thm:rig} does not need to be the same $\alpha$ as in~\eqref{eq:Wn}.

\item The drift vector field $a$ plays no role in the definition of $\phi_b$.
    This is expected since the driver $V_n(t)=n^{-1}\floor{tn}$ corresponding
    to $a$ in the RDE solved by $X_n$ (see the proof of Theorem~\ref{thm:rig} below)
    converges in $q$-variation for every $q>1$ to a process with no jumps.
    
    \item Since the limiting process $L$ in general has jumps,
    it is crucial that we pair $(L,X)$ with the path function $\phi_b$.
    In contrast, the jumps of $(W_n,X_n)$ are of magnitude at most $n^{-1/\alpha}$, so $(W_n,X_n)$ is almost a continuous path for large $n$;
    we make the reference to $\ell_{d+m}$ only for convenience (cf.~\eqref{eq:phi_to_linear} below).
\end{enumerate}
\end{rmk}

Recall that a stochastic process $(L_t)_{t \in[0,1]}$ is called \emph{stochastically continuous} if, for all $t\in[0,1]$, $L_s\to L_t$ in probability as $s\to t$.
Note that L{\'e}vy processes are stochastically continuous by definition.

\begin{cor}
    In the setting of Theorem~\ref{thm:rig}, suppose further that the process $L$ is stochastically continuous.  Then $X_n \to X$ in the sense of finite dimensional distributions.
\end{cor}
\begin{proof}
    Consider $0 \leq t_1 < \cdots < t_k \leq 1$.
    The map
    \begin{equation}\label{eq:time_proj_map}
        (Y, \phi) \mapsto (Y(t_1), \ldots, Y(t_k))
        \; ,
        \qquad
        (\sD^{p\var}([0,1],\R^{d+m}), \balpha_{p\var}) \to \R^{(d+m)k}
    \end{equation}
    is continuous at $(Y,\phi)$ whenever the path $Y$ is continuous at all $t_j$, see~\cite[Lemma~2.12]{CF19}.
    Furthermore, if $t\in[0,1]$ is a continuity point of $L$, then it is also a continuity point of the solution $X$ to~\eqref{eq:Marcus}.
    Since $L$ is c{\`a}dl{\`a}g and stochastically continuous, any fixed $t\in[0,1]$ is a.s.\ a continuity point of $L$ (see e.g.\ the proof of~\cite[Lemma~2.3.2]{Applebaum09}),
    $((L,X),\phi_b)$ is a.s.\ a continuity point of the map~\eqref{eq:time_proj_map}.
    In particular, by Theorem~\ref{thm:rig} and the continuous mapping theorem, $(X_n(t_1), \ldots, X_n(t_k))$ converges in
    law to $(X(t_1), \ldots, X(t_k))$, as required.
\end{proof}

\begin{rmk}
    As in Example~\ref{ex:circle}, we do not expect that $X_n \to_w X$ 
    in any of the Skorokhod topologies, or that $f(X_n) \to_w f(X)$ for
    certain standard functionals $f \colon D \to \R$ that are continuous with respect to the Skorokhod topologies,
    such as $f(X) = \|X\|_\infty$. Instead we have for example that $\|\tX_n\|_\infty \to_w \|\tX\|_\infty$,
    where $\tX_n$ and $\tX$ are the corresponding components of the continuous paths $(W_n, X_n)^{\ell_{d+m}, 1}$
    and $(W, X)^{\phi_b, 1}$.
\end{rmk}

\section{Rough path formulation}
\label{sec:rp}

In this section we expand the material in Section~\ref{sec:prep} in order to formulate and prove an abstract convergence result, Theorem~\ref{thm:CFW}, from which Theorem~\ref{thm:rig} follows.

 \subsection{Generalised \texorpdfstring{$\cSM_1$}{SM1} topologies with mixed variation}

We use a modified version of the topologies from~\cite{CF19} suitable for handling differential equations with drift.
We continue using notation from Section~\ref{sec:prep}.

For $1\le p,q<\infty$, we define the mixed $(q,p)$-variation for $u = (u^0,u^1,\ldots, u^d) = (u^0,\bar u) \colon [0,1]\to\R^{1+d}$ by
\[
\|u\|_{(q,p)\var} = \|u^0\|_{q\var} + \|\bar u\|_{p\var}\;.
\]
Let
\[
D^{(q,p)\var} = \{u\in D([0,1],\R^{1+d}) : \|u\|_{(q,p)\var}<\infty\}
\]
and $C^{(q,p)\var}([0,1],\R^{1+d})\subset D^{(q,p)\var}$ be the set of $u\in D^{(q,p)\var}$ which are continuous.
We furthermore denote
$\vvvert u \vvvert_{(q,p)\var} = |u(0)| + \|u\|_{(q,p)\var}$ and define
\[
  \bsigma_{(q,p)\var}(X_1, X_2) 
  = \inf_{\lambda\in\Lambda} \max \{
    \|\lambda - \id\|_\infty, \vvvert X_1 \circ \lambda - X_2 \vvvert_{(q,p)\var} 
  \}\;.
\]
Given $(X_1, \phi_1)$ and $(X_2, \phi_2)$ in $\bsD$, let
\[
  \balpha_\infty ((X_1, \phi_1), (X_2, \phi_2)) = \lim_{\delta \to 0} \bsigma_\infty (X_1^{\phi_1, \delta}, X_2^{\phi_2, \delta})\;.
\]
Following~\cite[Lemma~2.7]{CF19}, the limit exists, is independent of the choice
of the sequence $r_k$, and is invariant under reparametrisation of the path functions.
In particular, $\balpha_\infty$ induces a pseudometric on $\sD$.

For $(X, \phi) \in \sD([0,1],\R^{1+d})$, let
\[
  \|(X, \phi)\|_{(q,p)\var} = \|X^{\phi, 1}\|_{(q,p)\var}
  \;.
\]
As before, note that $\|(X, \phi)\|_{(q,p)\var}$ is well-defined since $\|X^{\phi, 1}\|_{(q,p)\var}$ does not depend on the parametrisation of $\phi$, nor the sequence
$\{r_k\}$. 
Let
\[
  \sD^{(q,p)\var} = \{
    (X, \phi) \in \sD : \|(X, \phi)\|_{(q,p)\var} < \infty
  \}
  \;.
\]
Given $(X_1, \phi_1)$ and $(X_2, \phi_2)$ in $\sD^{(q,p)\var}$, let
\[
  \\
  \balpha_{(q,p)\var} ((X_1, \phi_1), (X_2, \phi_2))
  = \lim_{\delta \to 0} \bsigma_{(q,p)\var} (X_1^{\phi_1, \delta}, X_2^{\phi_2, \delta})\;,
\]
which is well-defined and induces a metric on $\sD^{(q,p)\var}$ (cf.~\cite[Remark~3.8]{CF19}).

\subsection{Differential equations with c\`adl\`ag drivers}
\label{subsec:rough_int}

For $\beta,\gamma>0$, denote by $C^{\beta,\gamma}$ the space of all $b=(b^0,b^1,\ldots, b^d) \colon \R^m \to \R^{m \times (1+d)}$ such that
\[
\|b\|_{C^{\beta,\gamma}} = \|b^0\|_{C^{\beta}} + \max_{i=1,\ldots, d}\|b^i\|_{C^{\gamma}} < \infty\;.
\]
Suppose $1\leq q \leq p < 2$ and that $b \in C^{\beta,\gamma}$ with $\beta>q$ and $\gamma > p$ such that
\begin{equation}\label{eq:conditions_on_b}
\frac{\beta-1}{p} + \frac1q > 1 \quad \text{ and } \quad 
\frac{\gamma-1}{q} + \frac1p > 1 \;.
\end{equation}
\begin{rmk}
See~\cite[Remark~12.7]{FV10} for a discussion about condition~\eqref{eq:conditions_on_b}.
In our applications, we will consider $\beta>1$ and $\gamma>p$ as fixed, and $q=1+\kappa$ for $\kappa>0$ arbitrarily small.
In this case condition~\eqref{eq:conditions_on_b} is always attained by taking $\kappa$ sufficiently small, which explains why it does not appear in Theorem~\ref{thm:rig}.
\end{rmk}
Recall that under these conditions, if $W \in C^{(q,p)\var}([0,1],\R^{1+d})$, then the canonical RDE (in the Young regime)
\[
      \mrd X = b(X) \mrd W
     \;,\qquad
     X(0) = \xi \in \R^m
    \]
admits a unique solution $X\in C^{p\var}([0,1],\R^m)$.

For general $W\in D^{(q,p)\var}([0,1],\R^{1+d})$, consider the RDE
\begin{equation}
  \label{eq:RDE?}
  \mrd X = b(X) * \mrd W
  \;,\qquad
  X(0) = \xi
  \;.
\end{equation}
Here, $*$ stands for one of the different ways to interpret a differential equation in the presence of discontinuities, which in general result in different solutions $X$.
Two common choices
(considered in the case $q=p$ by Williams~\cite{W01} and studied further in~\cite{FS17, Chevyrev18, CF19, FZ18}) are
\begin{itemize}
  \item \emph{Geometric (Marcus) RDE.} The solution is completely analogous to that of~\eqref{eq:Marcus_W}: we solve the continuous RDE $\mrd \tX = b(\tX)\mrd W^{\phi,1}$, where $\phi=\ell_{1+d}$ is the linear path function on $\R^{1+d}$, and then remove the fictitious time intervals (note that the RDE is well-posed since $\|W^{\phi,1}\|_{(q,p)\var} \lesssim \|W\|_{(q,p)\var}$ by~\cite[Corollary~A.6]{Chevyrev18}).
    For geometric RDEs we use the notation
    \begin{equation}\label{eq:geometric_RDE}
      \mrd X = b(X) \diamond\mrd W
     \;,\qquad
     X(0) = \xi
      \;.
    \end{equation}
    Observe that $((W,X),\phi_b)\in \sD^{(q,p)\var}([0,1],\R^{1+d+m})$, where $\phi_b$ is the path function on $\R^{1+d+m}$ as in Definition~\ref{def:phi_b_def} with $\ell_{d}$ replaced by $\ell_{1+d}$.
  \item \emph{Forward (It{\^o}) RDE.} The solution satisfies the integral equation
\begin{equation}\label{eq:forward_integral}
X(t) = X(0) + \int_0^t b(X(s-)) \mrd W(s)\;,
\end{equation}
where the integral is understood as a limit of Riemann-Stieltjes sums with $b(X(s-))$ evaluated at the left limit points of the partition intervals:
\[
  \int_0^t b(X(s-)) \mrd W(s)
  = \lim_{|\cP| \to 0} \sum_{[s,s'] \in \cP} b(X(s-)) (W(s') - W(s))
  \; .
\]
Here, $\cP$ are partitions of $[0,t]$ into intervals, and $|\cP|$ is
the size of the longest interval.
    For forward RDEs we use the notation
    \[
      \mrd X = b(X)^- \mrd  W
     \;,\qquad
     X(0) = \xi
      \;.
    \]
\end{itemize}

\begin{rmk}\label{rmk:generalised_RDEs}
Geometric RDEs use linear paths to connect the endpoints of each jump.
As mentioned in the introduction, this has been generalised in~\cite{CF19}
allowing one to solve
    \begin{equation}
      \label{eq:GGRDE}
      \mrd X = b(X) \diamond \mrd (W, \phi)
     \;,\qquad
     X(0) = \xi
      \;,
    \end{equation}
    for any $(W,\phi) \in \sD^{(q,p)\var}([0,1],\R^{1+d})$.
    The interpretation is as for geometric RDEs: we construct a continuous path,
solve the canonical RDE $\mrd \tX = b(\tX) \mrd W^{\phi,1}$,
    and then remove fictitious time intervals.
Then $((W,X), \phi_b) \in\sD^{(q,p)\var}([0,1], \R^{1+d+m})$, where $\phi_b$ is the path function on $\R^{1+d+m}$ as in Definition~\ref{def:phi_b_def} with $\ell_{d}$ replaced by $\phi$, and the solution map of~\eqref{eq:GGRDE}
\begin{align*}
  \R^m \times \bigl(\sD^{(q,p)\var}([0,1], \R^{1+d}), \balpha_{(q,p)\var}\bigr)
  & \to \bigl(\sD^{(q,p)\var}([0,1], \R^{1+d+m}), \balpha_{(q,p)\var}\bigr)
  \;, \\
  (\xi, (W, \phi))
  & \mapsto ((W, X), \phi_b)
\end{align*}
is locally Lipschitz continuous. (These results were shown in~\cite[Theorem~3.13]{CF19} for $q=p$, but the same proof applies mutatis mutandis for the general case upon using the RDE with drift estimates~\cite[Theorem~12.10]{FV10}.
In fact one can allow rough path drivers in $\R^{d'+d}$ with finite $(q,p)$-variation for arbitrary $p,q\geq 1$ satisfying $p^{-1}+q^{-1}>1$.
We consider only $d'=1$ and $1\leq q\leq p<2$ since this suffices for our purposes.)
\end{rmk}

\subsection{Convergence of forward RDEs to geometric RDEs}

For the remainder of this section, let us fix $1\leq q\leq p<2$, $\beta>q$, $\gamma>p$, such that~\eqref{eq:conditions_on_b} holds.
Suppose that $W \in D^{(q,p)\var}([0,1],\R^{1+d})$ and
$b \in C^{\beta,\gamma}$.
Then for every $\xi \in \R^m$, the geometric RDE
\[
    \mrd \tX = b(\tX) \diamond \mrd W
    \; , \qquad \tX(0) = \xi
\]
admits a unique solution $\tX\in D^{p\var}([0,1],\R^m)$.

Suppose now that $W$ has finitely many jumps
at times $0 < t_1 < \cdots < t_n \leq 1$.
Then the solution $X$ of the forward RDE
\[
    \mrd X = b(X)^- \mrd W
    \; , \qquad X(0) = \xi
\]
can be obtained by solving the canonical RDE on each of the intervals
$[0,t_1), [t_1,t_2),\ldots [t_n, 1)$ (on which $W$ is continuous), and requiring that at the jump times
\begin{equation}\label{eq:jumps_of_X}
    X(t_k) = X(t_k-) + b(X(t_k-)) (W(t_k) - W(t_k-))
    \; .
\end{equation}
Hence in the case that $W$ has finitely many jumps, it is straightforward to construct the solution $X$ first on $[0, t_1)$, then
at $t_1$, then on $[t_1, t_2)$ and so on.
As we shall see, this construction furthermore allows for an easy extension of stability results of continuous RDEs to the setting with jumps.

\begin{rmk}
    \label{rmk:infinite_jumps}
    The construction of the forward solution for processes with infinitely many
    discontinuities is more involved, and can be achieved by solving directly the integral equation~\eqref{eq:forward_integral}.
This is done in~\cite{FZ18} but is not required here.
\end{rmk}

Recall that $\phi_b$ is the path function on $\R^{1+d+m}$ as in Definition~\ref{def:phi_b_def} with $\ell_{d}$ replaced by $\ell_{1+d}$.

\begin{thm}
    \label{thm:CFW}
    Suppose that $\{W_n\}_{n \geq 1}$ is a sequence of
    $D^{(q,p)\var}([0,1], \R^{1+d})$-valued random elements with almost surely
    finitely many jumps.
    Suppose that $b\in C^{\beta,\gamma}$.
    Let $X_n$ be the solution of the forward RDE
    \[
        \mrd X_n = b(X_n)^- \mrd W_n
        \;, \qquad
        X_n(0) = \xi_n \in \R^m
        \;. 
    \]
    Suppose that
    \begin{enumerate}[label=(\alph*)]
    \item \label{Wn:init_cond} $\lim_{n\to\infty} \xi_n = \xi$ for some $\xi\in\R^m$,
    \item \label{Wn:conv}
            $W_n\to_w W$  in $D([0,1],\R^{1+d})$ with the
$\cSM_1$  topology 
            as $n \to \infty$ (we allow the limit process $W$ to have infinitely many jumps),
        \item \label{Wn:tight} the family of random variables $\|W_n\|_{(q,p)\var}$ is tight,
        \item \label{Wn:jumps}
            $\sum_t |W_n(t) - W_n(t-)\bigr|^2 \to_w 0$ as $n \to \infty$,
            where the sum is over all jump times of $W_n$.
    \end{enumerate}
    Then $\|W\|_{(q,p)\var} < \infty$ almost surely.
    Let $X$ be the solution of the geometric RDE
    \[
        \mrd X = b(X) \diamond \mrd W
        \;, \qquad
        X(0) = \xi
        \;.
    \]
    (The RDE is well-posed because $\|W\|_{(q,p)\var} < \infty$.)
    Then for each $q'>q$ and $p'>p$, 
    \[
        ((W_n, X_n), \ell_{1+d+m}) \to_w ((W, X), \phi_b)
        \qquad \text{in} \qquad
        \bigl(\sD([0,1],\R^{1+d+m}), \balpha_{(q',p')\var}\bigr)
    \]
as $n\to\infty$.
\end{thm}

We give the proof after several preliminary results.
We will see that if $X_n$ solved the geometric RDE $\mrd X_n = b(X_n) \diamond \mrd W_n$ instead of the forward RDE, then Theorem~\ref{thm:CFW} would readily follow from~\cite{CF19} (and assumption~\ref{Wn:jumps} would not be needed).
In Lemma~\ref{lem:aoit}, we verify that under assumption~\ref{Wn:jumps}
the solution of the forward RDE $\mrd X_n = b(X_n)^- \mrd W_n$ closely approximates
the solution of the geometric RDE $\mrd X_n = b(X_n) \diamond \mrd W_n$
(generalising a result of~\cite{W01}).
First we show how a single jump of a geometric solution
relates to a ``forward'' jump (cf.~\cite[Lemma~1.1, Eq.~(11)]{W01}).
Define the semi-norm
\[
\|b\|_{\Lip} = \sup_{x,y\in\R^m} \frac{|b(x)-b(y)|}{|x-y|}\;.
\]

\begin{lemma}
    \label{lem:Wi}
    Suppose that $X \in C([0,1], \R^m)$ solves the ODE
    $\mrd X = b(X) \mrd t$ with $b$ Lipschitz. Then
    $\bigl|X(1) - X(0) - b(X(0))\bigr| \leq \|b\|_{\Lip}\|b\|_\infty / 2$.
\end{lemma}

\begin{proof}
    Write
        $X(1)
        = X(0) + b(X(0))
          + \int_0^1 \bigl(b(X(t)) - b(X(0))\bigr) \mrd t$.
    Since $|X(t) - X(0)| \leq \|b\|_\infty t$, 
    \[
        \Bigl| \int_0^1 \bigl(b(X(t)) - b(X(0))\bigr) \mrd t \Bigr|
        \leq \|b\|_{\Lip} \int_0^1 |X(t) - X(0)| \mrd t
        \leq \|b\|_{\Lip}\|b\|_{\infty} \int_0^1 \, t \mrd t
        \; .\qedhere
    \]
\end{proof}

%

We now quantify the error in moving from forward to geometric solutions.

\begin{lemma}
  \label{lem:aoit}
  Suppose that $W \in D^{(q,p)}([0,1], \R^{1+d})$ has finitely many jumps.
  Let $b \in C^{\beta,\gamma}$
  and let $X, \tX \in D([0,1], \R^m)$ be given by
\[
    \mrd X   = b(X)^- \mrd W\;,    \qquad      
    \mrd \tX = b(\tX) \diamond \mrd W\;,    \qquad      X(0)=\tX(0) = \xi\;.
\]
  Then
  \[
    \|X - \tX\|_{p\var}
    \leq \|b\|_{\Lip}\|b\|_\infty K
    \sum_{t} |W(t) - W(t-)|^2
    \;,
  \]
  where $K$ depends only on $\|b\|_{C^{\beta,\gamma}}$, $\|W\|_{(q,p)\var}$, $\gamma$, $\beta$, $p$, and $q$, and
  the sum is over all jump times $t$ of $W$.
\end{lemma}

\begin{proof}
  Let $t_1 < \cdots < t_n$ be the jump times of $W$; let $t_0 = 0$.
  For $j \leq n$, define $X_j$ as the solution of
  forward RDE $\mrd X_j = b(X_j)^- \mrd W$, $X_j(0) = \xi$,
  on $[0, t_j]$, and as the solution of the geometric RDE
  $\mrd X_j = b(X_j) \diamond \mrd W$ on $[t_j, 1]$
  with the initial condition taken from the solution on $[0, t_j]$.
  
  For each $j$, the processes $X_{j-1}$ and $X_j$ coincide on
  $[0, t_j)$ but possibly differ at $t_j$.
  By Lemma~\ref{lem:Wi} and the identity~\eqref{eq:jumps_of_X},
  \begin{equation}
    \label{eq:hha8}
    |X_j(t_j) - X_{j-1}(t_j)|
    \leq \frac{1}{2} \|b\|_{\Lip}\|b\|_\infty |W(t_j) - W(t_j-)|^2
    \;.
  \end{equation}
  On $[t_j, 1]$, both $X_{n, j-1}$ and $X_{n, j}$ solve the geometric RDE
  $\mrd X = b(X) \diamond \mrd W$, although with possibly different
  initial conditions. Recall that solutions of geometric RDEs are obtained
  from RDEs driven by continuous paths by inserting fictitious time intervals
  and linearly bridging the jumps.
  As such, they enjoy Lipschitz dependence on the initial condition (see~\cite[Theorem~12.10]{FV10})
  \begin{equation}
      \label{eq:hha9}
      \begin{aligned}
          \vvvert X_j - X_{j-1} \vvvert_{p\var;[t_j,1]}
          & = | X_j(t_j) - X_{j-1}(t_j) |
            + \| X_j - X_{j-1} \|_{p\var;[t_j,1]}  
       \\ & \leq K
          | X_j(t_j) - X_{j-1}(t_j) |
          \; ,
      \end{aligned}
  \end{equation}
  where $K$ depends only on $\|b\|_{C^{\beta,\gamma}}$, $\|W\|_{(q,p)\var}$, $\gamma$, $\beta$, $p$, and $q$.
  
  It follows from~\eqref{eq:hha8} and~\eqref{eq:hha9} that
  \[
    \vvvert X_j - X_{j-1} \vvvert_{p\var}
    \leq \frac{1}{2} \|b\|_{\Lip}\|b\|_\infty K
    |W(t_j) - W(t_j-)|^2
    \; .
  \]
  Observing that $X_0 = \tX$ and $X_n = X$, and taking the sum over $j$,
  we obtain the result.
\end{proof}
    
\begin{proof}[Proof of Theorem~\ref{thm:CFW}]
Denote by $\balpha_{(q,p)\var}$ the metric on $D^{(q,p)\var}([0,1],\R^k)$ induced by the
    corresponding metric on $\sD^{(q,p)\var}([0,1],\R^k)$ upon pairing paths with the linear path function $\ell_{k}$, i.e.\
    $\balpha_{(q,p)\var}(X_1,X_2) = \balpha_{(q,p)\var}((X_1, \ell_{k}), (X_2, \ell_{k}))$.
Let $D^{0,(q,p)\var} \subset D^{(q,p)\var}$ denote the closure of smooth paths in $(D^{(q,p)\var}, \balpha_{(q,p)\var})$.
By the same argument as~\cite[Proposition~3.10~(v)]{CF19}, note that $D^{(q,p)\var} \subset D^{0,(q',p')\var}$ for all $q' > q$ and $p' > p$.

    Fix $1\leq q' \leq p'<2$ with $p' \in (p,\gamma)$, $q'\in(q,\beta)$, and such that~\eqref{eq:conditions_on_b} holds with $q,p$ replaced by $q',p'$.
By~\cite[Proposition~2.9]{CF19}, convergence in $\cSM_1$ is equivalent to convergence in
    $(D, \balpha_\infty)$.
    By the Skorokhod representation theorem, we can thus suppose that a.s.\ $\lim_{n\to\infty}\balpha_\infty(W_n, W)=0$.
    Tightness of $\{\|W_n\|_{(q,p)\var}\}$ implies that a.s.\ there is a subsequence $n_k$ such that $\limsup_{k \to \infty} \|W_{n_k}\|_{(q,p)\var} < \infty$, and thus $\|W\|_{(q,p)\var}<\infty$ a.s.\ by lower semi-continuity of $(q,p)$-variation.
In addition, by a standard interpolation argument (cf.~\cite[Lemma~3.11]{CF19}), it holds that $\balpha_{(q',p')\var}(W_n,W) \to 0$ in probability, and therefore $W_n \to_w W$ in $(D^{0,(q',p')\var}, \balpha_{(q',p')\var})$.

Since $(D^{0,(q',p')\var}, \balpha_{(q',p')\var})$ is separable, we can again apply the Skorokhod representation theorem and suppose henceforth that, a.s., $W_n \to W$ in $\balpha_{(q',p')\var}$ and $\sum |W_n(t) - W_n(t-)|^2 \to 0$ (we used here that $\sum |W_n(t) - W_n(t-)|^2$ converges in law to a constant).

    An application of the continuity of solution map for generalised geometric
    RDEs (the proof of~\cite[Theorem~3.13]{CF19} combined with~\cite[Theorem~12.10]{FV10}; see Remark~\ref{rmk:generalised_RDEs}) shows that
   \begin{equation}\label{eq:W_n_to_W}
   ((W_n, \tX_n), \phi_b) \to ((W, X), \phi_b) \; \text{ in } \; (\sD^{(q',p')\var}([0,1],\R^{1+d+m}), \balpha_{(q',p')\var})\;,
\end{equation}
    where $\tX_n$ solves the geometric RDE
    \[
        \mrd \tX_n = b(\tX_n) \diamond \mrd W_n
        \;, \qquad
        \tX_n(0) = \xi_n
        \;. 
    \]
    Furthermore, since clearly
    \begin{equation}\label{eq:phi_to_linear}
    \lim_{n\to\infty}\balpha_\infty(((W_n,\tX_n),\phi_b),((W_n,\tX_n),\ell_{1+d+m})) = 0\;,
    \end{equation}
    it follows from~\cite[Lemma~3.11]{CF19} that
    \begin{equation}\label{eq:phi_b_to_phi}
    \lim_{n \to \infty} \balpha_{(q',p')\var}(((W_n,\tX_n),\phi_b),((W_n,\tX_n),\ell_{1+d+m})) = 0\;.
    \end{equation}
    It follows from Lemma~\ref{lem:aoit} that $\lim_{n\to\infty}\|(W_n,\tX_n)-(W_n,X_n)\|_{p'\var} = 0$,
    and in particular that $\bsigma_\infty((W_n,\tX_n),(W_n,X_n)) \to 0$.
    By virtue of interpolation, for each $q''>q'$ and $p''>p'$, the identity map
    \[
        (W,X) \mapsto ((W,X),\ell_{1+d+m})
        \;, \qquad
        (D^{(q',p')\var},\bsigma_\infty) \to (D^{(q',p')\var},\balpha_{(q'',p'')\var})
    \]
    is uniformly continuous on sets bounded in $(q',p')$-variation
    (cf.~\cite[Proposition~3.12]{CF19}),
    from which it follows that
    \begin{equation}\label{eq:tX_toX}
    \lim_{n\to\infty}\balpha_{(q'',p'')\var}(((W_n,\tX_n),\ell_{1+d+m}),((W_n,X_n),\ell_{1+d+m})) = 0\;.
    \end{equation}
Combining~\eqref{eq:W_n_to_W},~\eqref{eq:phi_b_to_phi}, and~\eqref{eq:tX_toX}, we obtain
\begin{equation*}
\lim_{n\to\infty} \balpha_{(q'',p'')\var}(((W_n,X_n),\ell_{1+d+m}, ((W,X),\phi_b)) = 0\;.
\end{equation*} 
    Since $q''>q'>q$ and $p''>p'>p$ are arbitrary, the conclusion follows.	
\end{proof}

We are now ready for the proof of Theorem~\ref{thm:rig}.

\begin{proof}[Proof of Theorem~\ref{thm:rig}]
Defining the process $V_n \colon [0,\infty) \to [0,\infty)$,
$V_n(t)=n^{-1}\floor{tn}$, observe that $X_n$ solves the forward RDE
\[
 \mrd X_n = a(X_n)^- \mrd V_n + b(X_n)^- \mrd W_n\;.
\]
       It follows from our assumptions that
        \begin{equation}\label{eq:rig:conv}
        (V_n,W_n) \to (\id,L) \quad \text{in the $\cSM_1$ topology}
        \end{equation}
        and 
        \begin{equation}\label{eq:rig:tight}
        \{\|(V_n,W_n)\|_{(1,p)\var}\}_{n \geq 1} \quad \text{is tight for every $p > \alpha'$}\;.
        \end{equation}
    Furthermore, since $\alpha < 2$ and $W_n$ makes at most $n$ jumps of size at most
        $n^{-1/\alpha} \|v\|_\infty$,
        \begin{equation}\label{eq:rig:jumps}
            \sum_t |W_n(t) - W_n(t-)\bigr|^2
            \leq \|v\|_\infty^2 n^{1-2/\alpha}
            \to 0
            \quad \text{as } n \to \infty
            \;.
        \end{equation}
        Choose $p \in (\alpha',\gamma)$ and $q \in (1, \min\{p, \beta\})$
        such that~\eqref{eq:conditions_on_b} is satisfied.
        By Theorem~\ref{thm:CFW}, it follows from~\eqref{eq:rig:conv},~\eqref{eq:rig:tight}, and~\eqref{eq:rig:jumps} that $\|L\|_{p\var}<\infty$ a.s.\ and
        \begin{equation}
            \label{eq:bouq}
            ((V_n,W_n, X_n), \ell_{1+d+m}) \to_w ((\id,L, X),\phi_{(a,b)})
        \end{equation}
in $(\sD^{(q,p)\var}([0,1],\R^{1+d+m}), \balpha_{(q,p)\var})$.
        Moreover, $\lim_{n\to0}\|V_n - \id\|_{q\var} =0$ and thus~\eqref{eq:bouq}
        readily implies that $((W_n, X_n), \ell_{d+m}) \to_w ((L, X),\phi_{b})$
        in $(\sD^{p\var}([0,1],\R^{d +m}), \balpha_{p\var})$.
\end{proof}

\section{Results for Gibbs-Markov maps}
\label{sec:GM}

In this section, we prove results on weak convergence to a L\'evy process, and tightness in $p$-variation, for a class of uniformly expanding maps known as
Gibbs-Markov maps~\cite{AaronsonDenker01}.
The weak convergence result extends work of~\cite{AaronsonDenker01,KPZ18,MZ15,TyranKaminska10} from scalar-valued observables to $\R^d$-valued observables.
The result on tightness in $p$-variation is new even for $d=1$.

\subsection{Gibbs-Markov maps}

Let $(Z,d)$ be a bounded metric space with Borel sigma-algebra $\cB$
and finite Borel measure $\nu$,
and an at most countable partition $\cP$ of $Z$ (up to a zero measure set)
        with $\nu(a) > 0$ for each $a \in \cP$.
Let $F\colon Z\to Z$ be a nonsingular ergodic measurable transformation.
We assume that $F$ is a {\em Gibbs-Markov map}. That is,
there are constants $\lambda > 1$, $K > 0$ and $\theta \in (0,1]$
such that for all $z,z'\in a$ and $a \in \cP$:
\begin{itemize}
    \item $Fa$ is a union of partition elements and $F$ restricts to a (measure-theoretic) bijection from $a$ to $Fa$;
        moreover $\inf_{a\in\cP} \nu(Fa)>0$; 
    \item $d(Fz, Fz') \geq \lambda d(z,z')$;
    \item the inverse Jacobian $\zeta_a = \frac{d\nu}{d\nu \circ F}$ of
        the restriction $F \colon a \to Fa$ satisfies
        \begin{equation}
            \label{eq:GM:dist}
            \bigl| \log \zeta_a(z) - \log \zeta_a(z') \bigr|
            \leq K d(Fz, Fz')^\theta
            \;.
        \end{equation}
\end{itemize}
It is standard (see for example \cite[Corollary p.~199]{AaronsonDenker01}) that there is
a unique $F$-invariant probability measure
$\mu_Z$ absolutely continuous with respect to $\nu$, with
bounded density $d\mu_Z / d\nu$.
The measure $\mu_Z$ is ergodic and we suppose for simplicity that $\mu_Z$ is mixing.  (The nonmixing case is also covered by standard arguments, see for example the end of the proof of~\cite[Proposition~4.3]{MZ15}, but is not required here.)

\begin{defn}
    \label{def:reg}
    We say that an $\R^d$-valued random variable $\xi$ is
    \emph{regularly varying} with index $\alpha > 0$ if
    there exists a probability measure $\sigma$
    on $\cB(\bbS^{d-1})$, the Borel sigma-algebra on the unit sphere
    $\bbS^{d-1} = \{x \in \R^d : |x| = 1 \}$, such that
    \[
        \lim_{t \to \infty}
        \frac{\PP(|\xi| > rt, \ \xi / |\xi| \in B)}{\PP(|\xi| > t)}
        = r^{-\alpha} \sigma(B)
    \]
    for all $r > 0$ and $B \in \cB(\bbS^{d-1})$ with $\sigma(\partial B) = 0$.
\end{defn}

    Recall that an $\alpha$-stable random variable $X$ in $\R^d$ with $\alpha\in(1,2)$ and $\E X = 0$
    has characteristic function 
    \[
        \E \exp( i u \cdot X)
        = \exp \biggl\{
            - \int_{\bbS^{d-1}} |u \cdot s|^\alpha
            \Bigl( 1 - i \sgn (u \cdot s) \tan \frac{\pi \alpha }{ 2} \Bigr)
            \mrd \Lambda(s)
        \biggr\}
        \; ,\quad
        u \in \R^d
        \; .
    \]
    Here $\Lambda$ is a finite nonnegative Borel measure on $\bbS^{d-1}$
with $\Lambda(\bbS^{d-1})>0$,
    known as the \emph{spectral measure} \cite[Section~2.3]{ST94}.
    It is a direct verification that $\gamma X$, with $\gamma \geq 0$,
    has spectral measure $\gamma^\alpha \Lambda$.

We say that an $\alpha$-stable L\'evy process $L_\alpha$ has spectral measure $\Lambda$ if
$L_\alpha(1)$ has spectral measure $\Lambda$.

Fix a function $\tau\colon Z\to\{1,2,\ldots\}$ that is constant on each $a\in\cP$ with value $\tau(a)$
such that $\int_Z \tau \mrd \mu_Z < \infty$.
Let $V\colon Z\to\R^d$ be integrable with $\int_Z V \mrd \mu_Z = 0$. Assume that
there exists $C_0>0$ such that for and all $z,z' \in a$, $a \in \cP$,
\begin{equation} \label{eq:C0}
    |V(z)| \le C_0\tau(a) \qquad\text{and}\qquad
    |V(z) - V(z')|
    \leq C_0 \tau(a)d(Fz,Fz')^\theta
    \; .
\end{equation} 

Suppose that $b_n$ is a sequence of positive numbers and
define the c\`adl\`ag process
\[
    W_n(t) = b_n^{-1} \sum_{j=0}^{\floor{nt}-1} V \circ F^j
    \; .
\]
We consider $W_n$ as a random element on the probability space $(Z, \mu_Z)$.
Throughout this section, $\|\cdot\|_p$ denotes the $L^p$ norm on $(Z,\mu_Z)$ for
$1 \le p \le \infty$ and $\E$ denotes expectation with respect to $\mu_Z$.

We now state the main results of this section.

\begin{thm} \label{thm:weakZ}
    Suppose that
    \begin{itemize}
        \item $V$ is regularly varying on $(Z, \mu_Z)$ with index $\alpha \in (1,2)$
            and $\sigma$ as in~Definition~\ref{def:reg},
        \item $b_n$ satisfies $\lim_{n \to \infty} n \mu_Z( |V| > b_n ) = 1$,
        \item $V - \E (V \mid \cP) \in L^p$ for some $p > \alpha$.
    \end{itemize}
    Then $W_n \to_w L_\alpha$ in the $\cSJ_1$ topology as $n \to \infty$,
    where $L_\alpha$ is the $\alpha$-stable L\'evy process
    with spectral measure $\Lambda = \cos \frac{\pi \alpha}{2} \Gamma(1-\alpha) \sigma$.
\end{thm}

\begin{rmk}\label{rmk:weakZ}\
    \begin{enumerate}[label=(\alph*)]
        \item\label{rmk:weakZ:reg}
            If $V$ is regularly varying and $\lim_{n \to \infty} n \mu_Z( |V| > b_n ) = 1$,
            then $b_n$ is a regularly varying sequence. In particular, if $\mu_Z( |V| > n) \sim c n^{-\alpha}$
            for some $c > 0$,
            then $b_n \sim c^{1/\alpha} n^{1/\alpha}$.
        \item\label{rmk:weakZ:Lp}
            In many examples (including the intermittent maps in Section~\ref{sec:PM}),
            $\tau \in L^q$ for each $q < \alpha$, and
            there exist $C > 0$ and $\beta \in (0,1)$ such that
            $|V(z) - V(z')| \leq C \tau^\beta$ for all $z, z' \in a$, $a \in \cP$.
            This implies that $V - \E (V \mid \cP) \in L^p$ for some $p > \alpha$.
    \end{enumerate}
\end{rmk}

\begin{thm} \label{thm:tightZ}
    Suppose that $\tau$ is regularly varying with index $\alpha \in (1,2)$
    on $(Z, \mu_Z)$, and that $b_n$ satisfies $\lim_{n\to\infty} n\mu_Z(\tau>b_n)=1$.
    Then $\sup_n\int_Z \|W_n\|_{p\var} \mrd\mu_Z < \infty$ for all $p>\alpha$.
\end{thm}

\subsection{Preliminaries about Gibbs-Markov maps}

We recall the following standard result.
\begin{lemma}
    \label{lem:Vdecomp}
Let $V \colon Z\to\R^d$ be integrable with $\int_Z V \mrd \mu_Z = 0$
and satisfying~\eqref{eq:C0}.  Then
    \begin{enumerate}[label=(\alph*)]
        \item\label{lem:Vdecomp:chi}
            $V=m+\chi \circ F - \chi$, where $m$ is integrable with $\E(m \mid F^{-1}\cB)=0$, and
            $\|\chi\|_\infty \leq C C_0$ with $C > 0$ independent of $V$.
        \item\label{lem:Vdecomp:p}
            For every $p \in (1,2]$ there is a constant $C(p)$, depending only on $p$, such that
            \[
                \biggl\|\max_{k \leq n} \Bigl| \sum_{j=0}^{k-1} V \circ F^j \Bigr| \biggr\|_p
                \leq C(p) n^{1/p} (\|\chi\|_\infty + \|V\|_p)
                \; .
            \]
            (We do not exclude the case $\|V\|_p = \infty$.)
    \end{enumerate}
\end{lemma}

\begin{proof}
    For $z, z' \in Z$, let $s(z, z')$ be the \emph{separation time}, i.e.\
    the minimal nonnegative integer such that
    $F^{s(z, z')}(z)$ and $F^{s(z, z')}(z')$ belong to different elements
    of $\cP$.
    Let $d_\theta$ be the separation metric on $Z$:
    \[
        d_\theta(z, z')
        = \lambda^{- \theta s(z, z')}
        \; .
    \]
    Note that $d(z, z')^\theta \leq d_\theta(z, z') (\diam Z)^\theta$,
    so $\theta$-H\"older observables with respect to $d$
    are $d_\theta$-Lipschitz.
    For an observable $\phi \colon Z \to \R^d$, let
    \[
        \|\phi\|
        =  \|\phi\|_\infty + 
        \sup_{z \neq z' } \frac{|\phi(z) - \phi(z')|}{d_\theta(z,z')}
        \; .
    \]
    
    Let $P \colon L^1(\mu_Z) \to L^1(\mu_Z)$ be the transfer operator
    corresponding to $F$ and $\mu_Z$, i.e.\
    \(\int_Z P \phi\,w \mrd\mu_Z = \int_Z \phi\,w\circ F \mrd\mu_Z\)
    for all $\phi\in L^1$, $w\in L^\infty$.
    By for example~\cite[Section~1]{AaronsonDenker01},
    there are constants $C_1>0$, $\gamma\in(0,1)$ such that
    $\|P^k \phi\| \leq C_1 \gamma^k \|\phi\|$
    for all $\phi \colon Z \to \R^d$ with $\E \phi = 0$ and
    all $k \geq 0$.

    By~\cite[Lemma~2.2]{MN05},
    there is a constant $C_2>0$ independent of $V$ such that
    $\|PV\| \le C_0 C_2$ for all $V$ satisfying
    the stated conditions.
    Hence 
    \[
        \|P^k V\|
        = \|P^{k-1}P V\|
        \leq  C_1 \gamma^{k-1}\|PV\|
        \leq C_0 C_1 C_2 \gamma^{k-1}
        \; .
    \]
    
    Let $\chi = \sum_{k=1}^\infty P^k V$.
    Then $\|\chi\|_\infty\le \|\chi\| \le C_0C_1C_2(1-\gamma)^{-1}$.
    Let $m = V - \chi \circ F + \chi$.
    Define $U \colon L^1(\mu_Z)\to L^1(\mu_Z)$ by $U\phi=\phi\circ F$.
    Then $PU=I$ and $UP=\E(\,\cdot \mid F^{-1}\cB)$.
    Hence $\E(m\mid F^{-1}\cB)=UPm=U(PV-\chi+P\chi)=0$ proving
    part~\ref{lem:Vdecomp:chi}.

    For part~\ref{lem:Vdecomp:p}, we proceed as in
    the proof of~\cite[Proposition~4.3]{MZ15}.
    Fix $n > 0$ and let $M^n_k = \sum_{j=n-k}^{n-1} m \circ F^j$.
    By~\ref{lem:Vdecomp:chi}, $M^n_k$ is a martingale on $0 \leq k \leq n$.
    By Burkholder's inequality, there is a constant $C(p)$ depending only on $p$
    such that
    \[
        \bigl\| \max_{k \leq n} |M^n_k| \bigr\|_p
        \leq C(p) n^{1/p} \|m\|_p
        \leq C(p) n^{1/p} ( 2 \|\chi\|_\infty + \|V\|_p)
        \; .
    \]
    Next,
    \[
        \biggl\|
            \max_{k \leq n} \Bigl|
                \sum_{j=0}^{k-1} V \circ F^j
            \Bigr|
        \biggr\|_p
        \leq 2 \|\chi\|_\infty + 2 \bigl\| \max_{k \leq n} |M^n_k| \bigr\|_p
        \; ,
    \]
    and part~\ref{lem:Vdecomp:p} follows.
%
\end{proof}

For sigma-algebras $\cF$ and $\cG$ on a common probability space $(\Omega,\PP)$, define
\begin{equation*}
    \psi(\cF, \cG)
    = \sup \biggl\{
        \frac{\bigl|\PP(A \cap B) - \PP(A) \PP(B) \bigr|}{\PP(A)\PP(B)}
        : A \in \cF, \ B \in \cG
    \biggr\}\;.
\end{equation*}
For $0\leq n \leq k$, let $\cP_n^k$ be the smallest sigma-algebra which contains $F^{-j} \cP$ for $j=n,\ldots, k$.
A standard property of mixing Gibbs-Markov maps (see for example~\cite[Section~1]{AaronsonDenker01})
is that there exist $\gamma\in(0,1)$ and $C>0$ such that for all $k \geq 0$,
$n \geq 1$,
\begin{equation} \label{eq:mixing}
    \psi(\cP_0^k, \cP_{n+k}^\infty)
    \leq C\gamma^n\;,
\end{equation}
where the probability measure in the definition of $\psi$ is $\mu_Z$. 

\subsection{Weak convergence to a L\'evy process}

In this subsection, we prove Theorem~\ref{thm:weakZ}.
We use the following result due to Tyran-Kami{\'n}ska~\cite{TyranKaminska10Rd}.

\begin{thm}
    \label{thm:MaRd}
    Let $X_0, X_1, \ldots$ be a strictly stationary sequence of integrable $\R^d$-valued
    random variables with $\E X_0 = 0$. For $0\le n \leq k$, let $\cF_n^k$ denote the sigma-algebra
    generated by $\{X_n, \ldots, X_k\}$. Suppose that:
    \begin{enumerate}[label=(\alph*)]
        \item\label{thm:MaRd:reg}
            $X_0$ is regularly varying with index $\alpha \in [1,2)$
            and $\sigma$ as in~Definition~\ref{def:reg}.
        \item\label{thm:MaRd:mix}
            $\sum_{j \geq 0} \psi(2^j) < \infty$,
            where
            $\psi(n) = \sup_{k\geq 0} \psi(\cF_0^k, \cF_{n+k}^\infty)$.
        \item\label{thm:MaRd:LD} 
            $\lim_{n \to \infty}
            \PP \bigl(
                |X_j| > \eps b_n
                \ \big\vert \ 
                |X_0| > \eps b_n
            \bigr)
            = 0$
            for all $\eps > 0$ and $j \geq 1$,
            where the sequence $b_n$ is such that
            $\lim_{n \to \infty} n \PP(|X_0| > b_n) = 1$.
    \end{enumerate}
    Then as $n \to \infty$, the random process $W_n$ given by
    $W_n(t) = b_n^{-1} \sum_{j = 0}^{\floor{nt}-1} X_j$
    converges to an $\alpha$-stable L\'evy process $L_\alpha$ in $D([0,1], \R^d)$ in
    the $\cSJ_1$ topology.
\end{thm}

\begin{rmk}
    \label{rmk:MaRd:law}
    It is implicit in~\cite{TyranKaminska10Rd} that $L_\alpha$ has spectral
    measure $\Lambda = \cos \frac{\pi \alpha}{2} \Gamma(1 - \alpha) \sigma$,
    where $\sigma$ is the measure on $\bbS^{d-1}$ for $X_0$ as in Definition~\ref{def:reg}.
\end{rmk}

\begin{proof}[Proof of Theorem~\ref{thm:MaRd}]
We verify the hypotheses of~\cite[Theorem~1.1]{TyranKaminska10Rd}.
In the notation of~\cite{TyranKaminska10Rd}, observe that~\ref{thm:MaRd:mix} and~\cite[Lemma~4.8]{TyranKaminska10Rd} together with $\rho\leq \psi$ imply that~\cite[Eq.~(1.6)]{TyranKaminska10Rd} holds.
Moreover,~\ref{thm:MaRd:LD} and~\cite[Corollary~1.3]{TyranKaminska10Rd} together with $\varphi\leq \psi$ imply that~\cite[\textbf{LD}($\phi_0$)]{TyranKaminska10Rd} holds (for inequalities concerning $\rho$, $\psi$, and $\varphi$, see~\cite{B05}).
\end{proof}

Write $V = V' + V''$ where $V' = \E(V \mid \cP)$. Let
\[
    W'_n(t) = b_n^{-1} \sum_{j=0}^{\floor{nt}-1} V' \circ F^j
    \; , \qquad
    W''_n(t) = b_n^{-1} \sum_{j=0}^{\floor{nt}-1} V'' \circ F^j
    \;.
\]

\begin{prop}
    \label{prop:tW''u}
    \
    \begin{enumerate}[label=(\roman*)]
        \item\label{prop:tW''u:Levy}
            $W'_n$ converges in $\cSJ_1$ to the $\alpha$-stable L\'evy process $L_\alpha$ 
            with spectral measure
            $\Lambda = \cos \frac{\pi \alpha}{2} \Gamma(1 - \alpha) \sigma$.
        \item\label{prop:tW''u:neg}
            $\bigl\|\sup_{t\in[0,1]}|W''_n(t)|\bigr\|_1 \to 0$ as $n \to \infty$.
    \end{enumerate}
\end{prop}

\begin{proof}
    To prove part~\ref{prop:tW''u:Levy}, we verify the hypotheses of
    Theorem~\ref{thm:MaRd} with $X_k = V' \circ F^k$.
    Since $\mu_Z$ is $F$-invariant, 
    $\{V' \circ F^k\}_{k \geq 0}$ is a strictly
    stationary sequence of $\R^d$-valued random variables.
    The remaining hypotheses are verified as follows
    \begin{itemize}
        \item[\ref{thm:MaRd:reg}]
            The observable $V$ is regularly varying with index $\alpha$ and measure $\sigma$, and
            $V'' \in L^p$ with $p > \alpha$, so $V'=V-V''$ is regularly
            varying with the same $\alpha$ and $\sigma$.
        \item[\ref{thm:MaRd:mix}] This is a consequence of~\eqref{eq:mixing}.
        
        \item[\ref{thm:MaRd:LD}]
            It follows from~\eqref{eq:mixing} and invariance of $\mu_Z$ under $F$
            that
            \[
                \mu_Z \bigl(
                    |V'\circ F^j| > \eps b_n
                    \ \big\vert \ 
                    |V'|          > \eps b_n
                \bigr)
                \lesssim  \mu_Z (|V'| > \eps b_n)
                \; .
            \]
    \end{itemize}
    
    Now we prove part~\ref{prop:tW''u:neg}.    
    By the assumptions of Theorem~\ref{thm:weakZ}, $V'' \in L^p$ for some $p \in( \alpha,2)$.
    Note that $|V''| \lesssim \tau$, $\E V'' = 0$ and for each $z,z' \in a$, $a \in \cP$,
    \[
        |V''(z) - V''(z')|=|V(z)-V(z')| \le C_0 \tau(a) d(Fz, Fz')^\theta
        \; .
    \]
    Hence by Lemma~\ref{lem:Vdecomp}\ref{lem:Vdecomp:p}, 
    \(
        \bigl\|
            \max_{k \leq n} |
                \sum_{j=0}^{k-1} V'' \circ F^j
            |
        \bigr\|_p
        \lesssim n^{1/p} = o(b_n)
    \).
\end{proof}

\begin{proof}[Proof of Theorem~\ref{thm:weakZ}]
    By Proposition~\ref{prop:tW''u}, 
    $W_n=W_n'+W_n'' \to_w L_\alpha$.
\end{proof}

\subsection{Tightness in \texorpdfstring{$p$}{p}-variation}

In this subsection we prove Theorem~\ref{thm:tightZ}.

First we record the following elementary properties of $\tau$.
(The Gibbs-Markov structure is not required here; the proof only uses that 
$\tau$ is regularly varying with values in $\{1,2,\ldots\}$ and that $\mu_Z$ is $F$-invariant.)

\begin{prop} \label{prop:tau}
    Let $p>\alpha$.
    Then
    \begin{enumerate}[label=(\alph*)]
        \item\label{prop:tau:p}
            $\E (\tau^p  1_{\{\tau\le b_n\}})=O(n^{-1}b_n^p)$,
        \item\label{prop:tau:1}
            $\E (\tau  1_{\{\tau\ge b_n\}})=O(n^{-1}b_n)$,
        \item\label{prop:tau:sum}
            $\E\big\{ \big(\sum_{j=0}^{n-1} \tau^p \circ F^j\big)^{1/p} \big\}
            =O(b_n)$.
    \end{enumerate}
\end{prop}

\begin{proof}
We have
    \[
        \E(\tau^p1_{\{\tau\le b_n\}})
        =\sum_{j\le b_n}j^p\mu_Z(\tau=j)
        \le \sum_{j\le b_n}(j^p-(j-1)^p)\mu_Z(\tau\ge j) \le
        p\sum_{j\le b_n}j^{p-1}\mu_Z(\tau\ge j) 
        \;.
    \]
    By Karamata's theorem~\cite[Proposition~1.5.8]{BGT}, $\E(\tau^p1_{\{\tau\le b_n\}})\lesssim b_n^p\mu_Z(\tau\ge b_n)$,
    so part~\ref{prop:tau:p} follows by definition of $b_n$. 
    A similar calculation proves part~\ref{prop:tau:1}.
    Next,
    \begin{align*}
        \Big(\sum_{j=0}^{n-1} \tau^p \circ F^j\Big)^{1/p}
        & \leq \Big(\sum_{j=0}^{n-1} \bigl(\tau^p 1_{\{\tau > b_n\}}\bigr) \circ F^j\Big)^{1/p}
        + \Big(\sum_{j=0}^{n-1} \bigl(\tau^p 1_{\{\tau \leq b_n\}} \bigr) \circ F^j\Big)^{1/p}
        \\ & \leq \sum_{j=0}^{n-1} \bigl(\tau 1_{\{\tau > b_n\}}\bigr) \circ F^j
        + \Big(\sum_{j=0}^{n-1} \bigl(\tau^p 1_{\{\tau \leq b_n\}} \bigr) \circ F^j\Big)^{1/p}
        \; .
    \end{align*}
    By Jensen's inequality, invariance of $\mu_Z$ and parts~\ref{prop:tau:p} and~\ref{prop:tau:1},
    \begin{align*}
        \E \Big\{\Big(\sum_{j=0}^{n-1} \tau^p \circ F^j\Big)^{1/p} \Big\}
        & \leq \sum_{j=0}^{n-1} \E \bigl\{ \bigl(\tau 1_{\{\tau > b_n\}}\bigr) \circ F^j \bigr\}
        + \Big(\sum_{j=0}^{n-1} \E \bigl\{ \bigl(\tau^p 1_{\{\tau \leq b_n\}} \bigr) \circ F^j \bigr\} \Big)^{1/p}
        \\
        & = n \E( \tau 1_{\{\tau > b_n\}})
        + \big( n \E( \tau^p 1_{\{\tau \leq b_n\}}) \big)^{1/p}
        \lesssim b_n
        \; ,
    \end{align*}
    proving part~\ref{prop:tau:sum}.
\end{proof}

Write 
$V=V_n'-\E V_n'+V_n''$, where
\[
  \SMALL
  V_n'=V1_{\{\tau > b_n\}}\;,
  \quad V_n''=V1_{\{\tau \le b_n\}}-\E( V1_{\{\tau \le b_n\}})\;.
\]
Accordingly, define $W_n=W_n'-\E W_n'+W_n''$, where
\[
    W_n'(t)=b_n^{-1}\sum_{j=0}^{\floor{nt}-1} V_n'\circ F^j\;, \qquad
    W_n''(t)=b_n^{-1}\sum_{j=0}^{\floor{nt}-1} V_n''\circ F^j\;.
\]

\begin{prop} \label{prop:Wn'}
$\sup_n\E \|W_n'\|_{1\var} < \infty$.
\end{prop}

\begin{proof}
    By Proposition~\ref{prop:tau}\ref{prop:tau:1},
    $\E|V_n'|\le C_0 \E\bigl(\tau 1_{\{\tau > b_n\}}\bigr) \lesssim n^{-1}b_n$.
Hence
    \[
\E\|W_n'\|_{1\var}=\E\Bigl(b_n^{-1}\sum_{j=0}^{n-1}|V_n'|\circ F^j\Bigr)
    = nb_n^{-1} \E|V_n'|=O(1)
\;,
\]
as required.
\end{proof}

\begin{prop} \label{prop:Wn''}  
    $\sup_n\E \|W_n''\|_{p\var}^p < \infty$ for all $p\in(\alpha, 2)$.
\end{prop}

\begin{proof}
Note that $\E V_n''=0$, that
$|V_n''|\le |V|+\E|V|\le C_1\tau$ where $C_1=C_0+\E|V|$, and that
$|V_n''(z)-V_n''(z')|\le |V(z)-V(z')|\le C_0\tau(a) d(Fz,Fz')^\theta$
for all $z,z'\in a$, $a\in\cP$.
By Lemma~\ref{lem:Vdecomp}\ref{lem:Vdecomp:chi},
    $V''_n = m_n + \chi_n \circ F - \chi_n$,
    where $ \sup_n \|\chi_n\|_\infty < \infty$
    and $\E(m_n \mid F^{-1} \cB) = 0$.
    Then
    \[
        \|m_n\|_p
        \le \|V_n''\|_p + 2\|\chi_n\|_p
        \le 2\|V 1_{\{\tau\le b_n\}}\|_p + 2\|\chi_n\|_\infty
    \]
    and
        $\E |V1_{\{\tau\le b_n\}}|^p 
        \le C_0^p\E\big( \tau^p 1_{\{\tau \leq b_n\}}\big)
        \lesssim n^{-1}b_n^p$
    by Proposition~\ref{prop:tau}\ref{prop:tau:p}.
    The assumptions of Theorem~\ref{thm:tightZ} imply that
    $b_n^p \gtrsim n$. Hence
    \begin{equation}
        \label{eq:glagla}
        \E |m_n|^p \lesssim n^{-1}b_n^p
        \; .
    \end{equation}

    Write $W_n''=M_n+B_n$ where
    \[
        M_n(t)=b_n^{-1}\sum_{j=0}^{\floor{nt}-1} m_n\circ F^j, \quad
        B_n(t)=b_n^{-1}\sum_{j=0}^{\floor{nt}-1} (\chi_n\circ F-\chi_n)\circ F^j
            =b_n^{-1}(\chi_n\circ F^{\floor{nt}}-\chi_n)
\;.
    \]
    Let
    \(
      M_n^-(t)
      = b_n^{-1} \sum_{j=1}^{\floor{nt}} m_n \circ F^{n-j}
      \;.
    \)
    Then $M_n^-$ is a martingale
    since $\E(m_n \mid F^{-1} \cB) = 0$.
    By~\cite[Theorem~2.1]{PisierXu88} and~\eqref{eq:glagla},
    \begin{equation}
        \label{eq:rtr}
        \E\|M_n\|_{p\var}^p
        = \E \|M_n^-\|_{p\var}^p
        \lesssim b_n^{-p} \sum_{j=1}^n \E |m_n \circ F^{n-j}|^p
        = nb_n^{-p} \E |m_n|^p
        \lesssim 1
        \;.
    \end{equation}
    Finally,
    $\|B_n\|_{p\var}^p
        \leq b_n^{-p}\, n\, (2\|\chi_n\|_\infty)^p\lesssim nb_n^{-p}\lesssim 1$
    for $p > \alpha$.
\end{proof}

\begin{rmk}
    For our purposes, it is sufficient to control the first moment $\E\|W_n''\|_{p\var}$.
    Hence we could have used the simpler result~\cite[Proposition~2]{Lepingle76} in place
    of the sharp result~\cite[Theorem~2.1]{PisierXu88}; this would give 
    $\sup_n\E\|W_n''\|_{p\var}^q<\infty$ for all $p > \alpha$ and $q < p$.
\end{rmk}

\begin{proof}[Proof of Theorem~\ref{thm:tightZ}]
Combine Propositions~\ref{prop:Wn'} and~\ref{prop:Wn''}.
\end{proof}


\section{Inducing weak convergence and tightness in \texorpdfstring{$p$}{p}-variation}
\label{sec:induce}

A general principle in smooth ergodic theory is that limit laws for dynamical systems are often inherited from the corresponding laws for a suitable induced system~\cite{Gouezel07,KM16,MT04,MZ15,Ratner73}.
In this section, we show that this principle applies to weak convergence in
$D([0,1],\R^d)$ with the $\cSM_1$ topology and to tightness in $p$-variation.
The results hold in a purely probabilistic setting.

Let $Y$ be a measurable space and $f \colon Y \to Y$ a measurable transformation.
Suppose that $Z\subset Y$ is a measurable subset
with a measurable return time
$\tau \colon Z \to \{1,2,\ldots\}$, i.e.\ $f^{\tau(z)}(z) \in Z$ for each $z \in Z$.
(It is not assumed that $\tau$ is the first return time.)
Define the induced map
\[
    F \colon Z \to Z
    \;, \qquad
    Fz = f^{\tau(z)}(z)
    \;.
\]
Suppose that $\mu_Z$ is an ergodic $F$-invariant probability measure
and that $\bar\tau=\int_Z\tau\mrd\mu_Z < \infty$.

Define the tower $f_\Delta \colon \Delta\to\Delta$
\begin{equation} \label{eq:tower}
\Delta = \{ (z, \ell) : z \in Z, 0 \leq \ell < \tau(z) \}
    \;, \qquad
    f_\Delta(z, \ell) = \begin{cases}
        (z, \ell + 1), & \ell < \tau(z) - 1 \;, \\
        (Fz, 0), & \ell = \tau(z) - 1 \;,
    \end{cases}
\end{equation}
with ergodic $f_\Delta$-invariant probability measure
$\mu_\Delta=(\mu_Z \times {\rm counting}) / \bar{\tau}$.
The map $\pi \colon \Delta \to Y$, $\pi(z,\ell)=f^\ell z$ defines a
measurable semiconjugacy between $f_\Delta$ and $f$, so
$\mu=\pi_*\mu_\Delta$ is an ergodic $f$-invariant probability measure on $Y$.

It is convenient to identify $Z$ with $Z \times \{0\} \subset \Delta$.
Then on the tower, $\tau$ is the first return time to $Z$.

Let $v \colon Y \to \R^d$ be measurable
and define
the corresponding \emph{induced observable}
\begin{equation} \label{eq:V}
    V \colon Z \to \R^d
    \;, \qquad
    V(z) = \sum_{j=0}^{\tau(z)-1} v(f^jz)
    \; .
\end{equation}

Let $v_k = \sum_{j=0}^{k-1} v \circ f^j$.
To measure how well the excursion $\{v_k(z)\}_{0 \leq k \leq \tau(z)}$
approximates the straight and monotone path from $0$ to $V(z)$, we define 
$V^* \colon Z \to \R^d$,
\begin{equation} \label{eq:Vstar}
    V^*
    = \inf_{c \in \R^d, |c| = 1}
    \Bigl(
        \max_{0 \leq k \leq \ell \leq \tau} c \cdot \bigl(v_k - v_\ell\bigr)
        + \max_{0 \leq k \leq \tau} \bigl| v_k - (c \cdot v_k) c\bigr|
    \Bigr)
    \; .
\end{equation}
Note that $V^*(z) = 0$ if and only if there exist
$0 = s_0 \leq s_1 \leq \cdots \leq s_{\tau(z)} = 1$ such that
$v_k(z) = s_k V(z)$ for $0 \leq k \leq \tau(z)$.

Let $b_n$ be a sequence of positive numbers, bounded away from 0, and define
\begin{equation} \label{eq:WntWn}
    W_n(t) = b_n^{-1} \sum_{j=0}^{\floor{nt}-1} v \circ f^j
    \qquad \text{and} \qquad
    \tW_n(t) = b_n^{-1} \sum_{j=0}^{\floor{nt}-1} V \circ F^j
    \; .
\end{equation}

In this section, the notation $\to_\mu$ and $\to_{\mu_Z}$ is used to
denote weak convergence for random variables defined on the probability
spaces $(Y, \mu)$ and $(Z, \mu_Z)$ respectively.
We prove:

\begin{thm}
    \label{thm:MZ:Rd}
    Suppose that
        $\tW_n \to_{\mu_Z} \tW$ in the $\cSM_1$ topology
    for some random process $\tW$.
    Suppose further that
    \[
        b_n^{-1} \max_{k<n}  V^* \circ F^k
        \to_{\mu_Z} 0
        \; .
    \]
Then
        $W_n \to_{\mu} W$ in the $\cSM_1$ topology
    where $W(t) = \tW(t / \bar \tau)$.
\end{thm}

\begin{thm} \label{thm:tight:induce}
    Suppose that $\tau$ is regularly varying with index $\alpha>1$ on $(Z, \mu_Z)$, and that
    $b_n$ satisfies $\lim_{n\to\infty} n\mu_Z(\tau>b_n)=1$.
    Let $v\in L^\infty$.
    Suppose that the family of random variables $\|\tW_n\|_{p\var}$ is tight on $(Z,\mu_Z)$ for some $p>\alpha$.
    Then the family $\|W_n\|_{p\var}$ is tight on $(Y, \mu)$.
\end{thm}

\begin{rmk}
    The assumptions of Theorem~\ref{thm:tight:induce} on $\tau$ can be relaxed. If
    $\tau' \colon Z \to \{1,2,\ldots\}$ is regularly varying with index $\alpha > 1$ on $(Z, \mu_Z)$
    and $b_n$ satisfies $\lim_{n\to\infty} n\mu_Z(\tau' > b_n)=1$,
    then the result holds for all $\tau \leq \tau'$.
\end{rmk}

\subsection{Inducing convergence in \texorpdfstring{$\cSM_1$}{SM1} topology}

In this subsection, we prove Theorem~\ref{thm:MZ:Rd}.
Our proof closely follows
the analogous proof in~\cite{MZ15}, with the difference that we
work in $\R^d$ instead of $\R$. 

Since $\pi:\Delta\to Y$ is a measure-preserving semiconjugacy, we may suppose without loss of generality that $Y = \Delta$ and $f=f_\Delta$ as in~\eqref{eq:tower}.
In particular, we may
suppose that $\tau$ is the first return time.

Define
\[
    u \colon Y \to \R^d
    \;, \qquad 
    u(y) = \begin{cases}
        V(z) \;, & y=(z,\tau(z)-1) \; , \\
        0\;, & \text{otherwise} \; .
    \end{cases}
\]
Let
\[
    U_n(t) = b_n^{-1} \sum_{j=0}^{\floor{nt}-1} u \circ f^j
    \; .
\]
Thus defined, the restriction of $U_n$ to $Z$ corresponds to
$U_n$ in \cite{MZ15}.

\begin{lemma}
    \label{lem:UntW}
    $U_n \to_{\mu_Z} W$ in the $\cSM_1$ topology.
\end{lemma}

\begin{proof}
For the case $d=1$, see \cite[Lemma~3.4]{MZ15}.
    The proof 
    for all $d \geq 1$ goes through unchanged.
\end{proof}

Next we control \emph{excursions:} we estimate the distance between $U_n$ and $W_n$
in the $\cSM_1$ topology.

\begin{prop}
    \label{prop:SM_1:bound}
    Let $w \in D([T_0,T_1], \R^d)$ and define
    $\phi \colon [T_0,T_1]\to \R^d$ to be the linear path with
    $\phi(T_0) = w(T_0)$ and $\phi(T_1) = w(T_1)$.
    Then for each $c \in \R^d$ with $|c| = 1$,
    \[
        d_{\cSM_1}(w, \phi)
        \leq T_1 - T_0 + 
        2 \sup_{T_0 \leq s < t \leq T_1} c \cdot w(t,s)
        + 2 \sup_{T_0 \leq t \leq T_1} \bigl| w(T_0, t) - (c \cdot w(T_0, t)) c\bigr|
        \; ,
    \]
    where $w(a,b) = w(b) - w(a)$.
\end{prop}

\begin{proof}
    Without loss of generality, we suppose that $w(T_0) = 0$.
    Define $\chi \colon [T_0, T_1] \to [0,\infty)$ and
    $\psi \colon [T_0, T_1]\to \R^d$ to be 
    $\chi(t) = \sup_{s \leq t} c \cdot w(s)$ and
$\psi(t) = \chi(t) c$.
    Then $\psi$ is a monotone path in the direction of $c$.

    Observe that
    \(
        |w(t) - \psi(t)|
        \leq \chi(t) - c \cdot w(t) + |w(t) - (c \cdot w(t)) c|
    \).
    Hence
    \begin{equation}
        \label{eq:w-psi}
        \sup_t|w(t)-\psi(t)|
        \leq \sup_{s < t} c \cdot w(t,s)
        + \sup_{t} | w(t) - (c \cdot w(t)) c |
        \; .
    \end{equation}

    Further, let $\xi \colon [T_0, T_1]\to \R^d$ be the linear path
    with
    $\xi(T_0) = w(T_0) = 0$ and $\xi(T_1) = \psi(T_1) = \chi(T_1) c$.
    Since $\xi$ is a reparametrisation of $\psi$ (up to linear jumps),
    \begin{equation}
        \label{eq:psi-xi}
        d_{\cSM_1}(\xi, \psi) \leq T_1 - T_0
        \; .
    \end{equation}
    Also, for each $\eps > 0$ there is $s \in [T_0, T_1]$ such that
    $|\chi(T_1) - c \cdot w(s)| \leq \eps$. Then
    \begin{equation}
        \label{eq:xi-phi}
        \begin{aligned}
            \sup_t|\phi(t) - \xi(t)|
            & = |\phi(T_1) - \xi(T_1)|
            \leq |w(T_1) - (c \cdot w(s)) c| + \eps
         \\ & \leq |w(T_1) - (c \cdot w(T_1)) c |
            + c \cdot (w(s) - w(T_1))
            + \eps
            \; .
        \end{aligned}
    \end{equation}
    The result follows from~\eqref{eq:w-psi},~\eqref{eq:psi-xi},~\eqref{eq:xi-phi}
    and that $\eps$ can be taken arbitrarily small.
\end{proof}

For $s \leq t$, let $d_{\cSM_1, [s,t]}$ denote the distance on $[s,t]$.
Let $\tau_k = \sum_{j=0}^{k-1} \tau \circ F$.

\begin{cor}
    \label{cor:UnWn}
    For each $n$ and $k$, on $Z$,
    \[
        d_{\cSM_1, [0, \tau_k / n]} (U_n, W_n)
        \leq 2 \max_{0\le j<k} \Bigl\{
            \frac{\tau \circ F^j}{n} + \frac{V^* \circ F^j}{b_n} 
        \Bigr\}
        \; .
    \]
\end{cor}

\begin{proof}
    Denote $T_j = \tau_j / n$. Since we restrict to $Z$, each interval $[T_j, T_{j+1}]$,
    including with $j=0$, corresponds to a complete excursion with
    $U_n(T_j) = W_n(T_j)$ and $U_n(T_{j+1}) = W_n(T_{j+1})$.
    Fix $j$ and let $\phi \colon [T_j, T_{j+1}]\to \R^d$ be the linear path such that
    $\phi(T_j) = U_n(T_j)$ and $\phi(T_{j+1}) = U_n(T_{j+1})$.
    Recall that $U_n$ is constant on $[T_j, T_{j+1})$.
    By Proposition~\ref{prop:SM_1:bound},
    \begin{align*}
        d_{\cSM_1, [T_j, T_{j+1}]}(U_n, \phi)
        & \leq T_{j+1} - T_j
        \; , \\
        d_{\cSM_1, [T_j, T_{j+1}]}(W_n, \phi)
        & \leq T_{j+1} - T_j + \frac{2}{b_n} V^* \circ F^j
        \; .
    \end{align*}
    Hence
    \[
        d_{\cSM_1, [T_j, T_{j+1}]} (U_n, W_n)
        \leq 2(T_{j+1} - T_j) + \frac{2}{b_n} V^* \circ F^j
        = \frac{2}{n}\tau\circ F^j + \frac{2}{b_n} V^* \circ F^j
        \; .
    \]
    Finally, 
    \[
        d_{\cSM_1, [0,T_k]} (U_n, W_n)
        \leq \max_{j < k} d_{\cSM_1, [T_j, T_{j+1}]}(U_n, W_n)
        \; ,
    \]
and the result follows.
\end{proof}

\begin{lemma}
    \label{lem:UnWnT}
    $d_{\cSM_1, [0, T]}(U_n, W_n) \to_{\mu_Z} 0$
    for all $T > 0$.
\end{lemma}

\begin{proof}
    Fix $T > 0$ and define the random variables $k = k(n) = \max \{ j \geq 0: \tau_j / n \leq T \}$ on $Z$.
    Consider the processes $U_n$, $W_n$ on $Z$,
    where the time interval $[0, \tau_k / n]$ corresponds to $k$ complete excursions,
    while $[\tau_k / n, T]$ is the final incomplete excursion.
    By Corollary~\ref{cor:UnWn} and the assumptions of Theorem~\ref{thm:MZ:Rd},
    \[
        d_{\cSM_1, [0,\tau_k / n]} (U_n, W_n)
        \leq 2 \max_{j<k} \Bigl\{
            \frac{\tau \circ F^j}{n} + \frac{V^* \circ F^j}{b_n} 
        \Bigr\}
        \to_{\mu_Z} 0
        \; .
    \]
    
    For $y=(z,\ell) \in Y$, let $E(y) = \sum_{j=0}^{\tau(z) - 1} \bigl|v(f^j z)\bigr|$.
    Since $\mu$ is $f$-invariant and $b_n \to \infty$, we have
    $b_n^{-1} E \circ f^{\floor{nT}} \to_{\mu} 0$. Since $\mu_Z$ is absolutely continuous
    with respect to $\mu$, we also have $b_n^{-1} E \circ f^{\floor{nT}} \to_{\mu_Z} 0$.
    Hence
    \[
\begin{aligned}
            d_{\cSM_1, [0,T]} (U_n, W_n)
            & \leq d_{\cSM_1, [0, \tau_k / n]} (U_n, W_n)
            + \sup_{[\tau_k / n, T]} |U_n- W_n|
            \\ & \leq d_{\cSM_1, [0,\tau_k / n]} (U_n, W_n)
            + \frac{1}{b_n} E \circ f^{\floor{nT}}\to_{\mu_Z} 0
\end{aligned}
    \]
as required.
\end{proof}

\begin{proof}[Proof of Theorem~\ref{thm:MZ:Rd}]
    By~Lemma~\ref{lem:UnWnT},
    $d_{\cSM_1, [0, T]}(U_n, W_n) \to_{\mu_Z} 0$ for every $T$.
    By Lemma~\ref{lem:UntW}, $U_n \to_{\mu_Z} W$ in $\cSM_1$.
    Hence $W_n \to_{\mu_Z} W$ in $\cSM_1$.
    The required convergence of $W_n \to_\mu W$ in $\cSM_1$
    follows from strong distributional convergence~\cite[Theorem~1]{Z07} upon verifying that
    $d_{\cSM_1}(W_n,W_n\circ f) \leq d_{\cSJ_1}(W_n,W_n\circ f) \to_\mu 0$
    in the same way as~\cite[Corollary~3]{Z07}.
\end{proof}

\subsection{Inducing tightness in \texorpdfstring{$p$}{p}-variation}
\label{sec:p-var}

In this subsection we prove Theorem~\ref{thm:tight:induce}.
Again, we suppose without loss of generality that $f:Y\to Y$ is the tower~\eqref{eq:tower}.

\begin{lemma}
  \label{lem:WZ:p}
  The family $\|W_n\|_{p\var}$ is tight on $(Z, \mu_Z)$.
\end{lemma}

\begin{proof}
Let $\tau_n=\sum_{j=0}^{n-1}\tau\circ F^j$ and
define $U_n(t) = b_n^{-1} \sum_{j=0}^{\floor{\tau_n t}-1} v \circ f^j$
on $Z$.
Note that $\|W_n\|_{p\var} \leq \|U_n\|_{p\var}$.
Let $s_i = \tau_i / \tau_n$, $i=0,\dots,n$ and
write $U_n = U'_n + U''_n$ where $U'_n|_{[s_i,s_{i+1})} \equiv U_n(s_i)$.
 
  Observe that $U'_n$ is a time-changed version of $\tW_n$
(indeed $U'_n(s_i)=\tW_n(i/n)$), so
  $\|U'_n\|_{p\var} = \|\tW_n\|_{p\var}$.
  Thus the family $\|U'_n\|_{p\var}$ is tight on $(Z, \mu_Z)$.
  
  Further we bound $\int_Z \|U''_n\|_{p\var} \mrd\mu_Z$.
Note that $U''_n(s_i)=0$ and
$\|1_{[s_i,s_{i+1})}U''_n\|_\infty\le b_n^{-1}\|v\|_\infty\tau\circ F^i$.
Hence for $t\in[s_i,s_{i+1})$,
$t'\in[s_{i'},s_{i'+1})$, 
\[
|U_n(t)-U_n(t')|^p\le\big(
b_n^{-1}\|v\|_\infty(\tau\circ F^i+\tau\circ F^{i'})\big)^p
\le 2^{p-1}b_n^{-p}\|v\|_\infty^p(\tau^p\circ F^i+\tau^p\circ F^{i'})
\;.
\]
It follows that
\[
\|U''_n\|_{p\var}^p\le \sum_{i=0}^{n-1}\|U''_n\|_{p\var,[s_i,s_{i+1}]}^p
+ 2^{p}b_n^{-p}\|v\|_\infty^p\sum_{i=0}^{n-1}\tau^p\circ F^i
\;.
\]
On $[s_i,s_{i+1}]$, there are $\tau\circ F^i-1$ jumps of size at most
$b_n^{-1}\|v\|_\infty$, and one jump of size at most $b_n^{-1} \|v\|_\infty \tau\circ F^i$, so
$\|U''_n\|_{p\var,[s_i,s_{i+1}]} \le \|U''_n\|_{1\var,[s_i,s_{i+1}]} \le 2b_n^{-1}\|v\|_\infty\tau\circ F^i$.
  Altogether, we have shown that
  \[
    \|U''_n\|_{p\var} 
    \lesssim \|v\|_\infty b_n^{-1} \Bigl(\sum_{j=0}^{n-1} \tau^p \circ F^j\Bigr)^{1/p}
    \;.
  \]
  Now apply Proposition~\ref{prop:tau}\ref{prop:tau:sum}.
\end{proof}

\begin{lemma}
    \label{lem:Zwei}
    The family $\|W_n\|_{p\var}$ is tight on $(Y, \mu_Z)$ if and only if it is tight on $(Y, \mu)$.
\end{lemma}

\begin{proof}
    Observe that
    $W_n(t) \circ f = W_n(t+{\SMALL \frac{1}{n}}) - b_n^{-1} v$ for all $t \geq 0$.
    Hence 
    \[
        \bigl| \|W_n\|_{p\var} - \|W_n\|_{p\var} \circ f \bigr|
        \leq b_n^{-1} (|v| + |v| \circ f^n) \to_\mu 0
        \;.
    \]
    Hence by~\cite[Theorem~1]{Z07}, 
    $\|W_{n_k}\|_{p\var}$ has the same limit in distribution (if any) on $(Y,\mu_Z)$ as on $(Y,\mu)$ 
    for each subsequence $n_k$.
    The result follows.
\end{proof}

\begin{proof}[Proof of Theorem~\ref{thm:tight:induce}]
    Combine Lemmas~\ref{lem:WZ:p} and \ref{lem:Zwei}.
\end{proof}

\section{Results for nonuniformly expanding maps}
\label{sec:driver}

In this section, we prove results on weak convergence to a L\'evy process, and tightness in $p$-variation, for a class of nonuniformly expanding maps.
The weak convergence result extends work of~\cite{MZ15} from scalar-valued observables to $\R^d$-valued observables.
The result on tightness in $p$-variation is again new even for $d=1$.

We show that intermittent maps such as~\eqref{eq:LSV} and~\eqref{eq:PM} fit
our setting in Subsection~\ref{sec:PM}.

\subsection{Nonuniformly expanding maps}
\label{sec:NUE}

Let $f\colon Y\to Y$ be a measurable transformation on a bounded metric space $(Y, d)$
and let $\nu$ be a finite Borel measure on $Y$.
Suppose that there exists a 
    Borel subset $Z\subset Y$ with $\nu(Z)>0$
    and an at most countable partition $\cP$ of $Z$ (up to a zero measure set)
        with $\nu(a) > 0$ for each $a \in \cP$.
    Suppose also that there is an integrable  \emph{return time} function $\tau \colon Z \to \{1,2,\ldots\}$
        which is constant on each $a \in \cP$ with value $\tau(a)$, such that
        $f^{\tau(a)}(z) \in Z$ for all $z \in a$, $a \in \cP$.

Define the \emph{induced map} $F \colon Z \to Z$, $F(z) = f^{\tau(z)}(z)$.
We assume that $f$ is \emph{nonuniformly expanding.} That is, $F$ is Gibbs-Markov as in Section~\ref{sec:GM} and in addition there is a constant $C>0$ such that
\begin{equation} \label{eq:NUE}
    d(f^k z, f^k z') \leq C d(Fz, Fz')
    \qquad
    \text{ for all $0 \leq k \leq \tau(a)$, $z,z'\in a$, $a\in\cP$}
    \; .
\end{equation}
Let $\mu_Z$ be the unique $F$-invariant probability measure absolutely continuous with respect to $\nu$.  Define
the ergodic $f$-invariant probability measure $\mu=\pi_*\mu_\Delta$ as in Section~\ref{sec:induce}.
Set $\bar\tau=\int_Z \tau\mrd\mu_Z$.

Let $v \colon Y \to \R^d$ be a H\"older
observable with $\int_Y v \mrd \mu = 0$,
and define
$V,\,V^* \colon Z \to \R^d$ as in~\eqref{eq:V} and~\eqref{eq:Vstar}.

Let $b_n$ be a sequence of positive numbers
and define $W_n$ as in~\eqref{eq:WntWn}.
Let $\PP$ be any probability measure on $Y$ that is absolutely continuous with respect to $\nu$, and regard $W_n$ as a process
with paths in $D([0,1], \R^d)$, defined on the probability space
$(Y, \PP)$.

We can now state and prove the main results of this subsection.

\begin{thm}
    \label{thm:Y}
    Suppose that:
    \begin{enumerate}[label=(\alph*)]
        \item $V\colon Z\to\R^d$ is regularly varying on $(Z, \mu_Z)$
            with index $\alpha \in (1,2)$ and $\sigma$ as in~Definition~\ref{def:reg}.
        \item $b_n$ satisfies $\lim_{n \to \infty} n \mu_Z( |V| > b_n ) = 1$.
        \item $V - \E(V \mid \cP) \in L^p$ for some $p > \alpha$,
            where $\E$ denotes the expectation on $(Z, \mu_Z)$.
        \item $b_n^{-1} \max_{k<n} V^* \circ F^k  \to_w 0$
            on $(Z, \mu_Z)$.
    \end{enumerate}
    Then $W_n \to_w L_\alpha$ on $(Y,\PP)$ in the $\cSM_1$ topology,
    where $L_\alpha$ is the $\alpha$-stable L\'evy process
    with spectral measure
    $
        \Lambda
        = \cos \frac{\pi \alpha}{2} \Gamma(1-\alpha) \sigma/\bar\tau
    $.
\end{thm}

\begin{proof}
    Note that $|V|\le \|v\|_\infty \tau$.  Let $z,z'\in a$, $a\in\cP$.  Then
    \[
        |V(z)-V(z')|\le \sum_{j=0}^{\tau(z)-1}|v(f^jz)-v(f^jz')|
        \le C_0
        \sum_{j=0}^{\tau(z)-1}d(f^jz,f^jz')^\theta
        \le C_0 \tau(a) d(Fz,Fz')^\theta
        \; ,
    \]
    where $C_0$ is the H\"older constant for $v$ and $\theta$ is the H\"older exponent,
    and we used condition~\eqref{eq:NUE} in the definition of nonuniformly expanding map.
    Hence condition~\eqref{eq:C0} is satisfied.

    Define $\tW_n$ as in~\eqref{eq:WntWn}.
    By Theorem~\ref{thm:weakZ},
    $\tW_n \to_w \tL_\alpha$ on $(Z, \mu_Z)$ in the $\cSJ_1$ topology where 
    $\tL_\alpha$ is an $\alpha$-stable L\'evy process with $\tL_\alpha$ having
    spectral measure $\tLambda = \cos \frac{\pi \alpha}{2} \Gamma(1-\alpha) \sigma$.

    By Theorem~\ref{thm:MZ:Rd}, $W_n \to_w L_\alpha$ on $(Y,\mu)$ in the $\cSM_1$ topology
    where $L_\alpha(t) = \tL_\alpha(t / \bar \tau)$.  This proves the result when $\PP=\mu$.

By~\cite[Theorem~1 and Corollary~3]{Z07} (see also~\cite[Proposition~2.8]{MZ15}), the convergence holds not only on $(Y,\mu)$ but also on $(Y,\PP)$ for any probability measure $\PP$ that is absolutely continuous with respect to 
$\nu$.
This completes the proof.
\end{proof}

\begin{thm}
    \label{thm:Yp}
    Suppose that $\tau$ is regularly varying with index $\alpha>1$ on $(Z, \mu_Z)$, and that
$b_n$ satisfies $\lim_{n\to\infty} n\mu_Z(\tau>b_n)=1$.
    Then $\{\|W_n\|_{p\var}\}$ is tight on $(Y, \PP)$ for each $p>\alpha$.
\end{thm}

\begin{proof}
Condition~\eqref{eq:C0} was established in the proof of Theorem~\ref{thm:Y}.
    Tightness on $(Y,\mu)$ follows from Theorems~\ref{thm:tight:induce} and~\ref{thm:tightZ}.
Tightness on $(Y,\PP)$ holds by the same argument used in the proof of Lemma~\ref{lem:Zwei}.
\end{proof}

\subsection{Intermittent maps}
\label{sec:PM}

In this subsection, we show that Theorems~\ref{thm:PM} and~\ref{thm:tight} hold for the intermittent maps
$f \colon [0,1] \to [0,1]$, given by~\eqref{eq:LSV} and~\eqref{eq:PM}.

We choose $Z = [\frac12,1]$ for the map~\eqref{eq:LSV}, and
$Z = [\frac13, \frac23]$ for~\eqref{eq:PM}. Let $\tau$ be the first return time to $Z$.
The reference measure $\nu$ is Lebesgue and
the partition $\cP$ consists of maximal intervals on which the return time
is constant.
It is standard that the first return map $F = f^\tau$ is Gibbs-Markov, and
since $f'>1$, condition~\eqref{eq:NUE} holds. Thus both maps
are nonuniformly expanding.

\begin{lemma}
    \label{lem:lsv}
    Let $v \colon [0,1] \to \R^d$ be H\"older with
    $\int v \mrd \mu = 0$ and
    $v(0) \neq 0$, also $v(1) \neq 0$
    in case $f$ is given by~\eqref{eq:PM}.
    Define $V,\,V^* \colon Z \to \R^d$ as in~\eqref{eq:V} and~\eqref{eq:Vstar}.
    Then
    \begin{enumerate}[label=(\alph*)]
        \item\label{lem:lsv:density}
            There exists a unique absolutely continuous $f$-invariant probability measure $\mu$ on $[0,1]$.
            Its density $h$ is bounded below and is continuous on $Z$.
        \item\label{lem:lsv:reg} $V$ is regularly varying with index $\alpha$
            on $(Z, \mu_Z)$. The probability measure $\sigma$
            as in Definition~\ref{def:reg} is given by
            \[
                \sigma =
                \begin{cases}
                    \delta_{v(0)/|v(0)|} & \text{ for the map \eqref{eq:LSV}} \; , \\[.75ex]
                    \frac{|v(0)|^\alpha}{|v(0)|^\alpha + |v(1)|^\alpha} \delta_{v(0)/|v(0)|} +
                    \frac{|v(1)|^\alpha}{|v(0)|^\alpha + |v(1)|^\alpha} \delta_{v(1)/|v(1)|}
                    & \text{ for the map \eqref{eq:PM}} \; .
                \end{cases}
            \]
        \item\label{lem:lsv:c} $\lim_{n \to \infty} n \mu_Z( |V| > b_n ) = 1$ with $b_n = c^{1 / \alpha} n^{1 / \alpha}$, where
            \[
                c = 
                \begin{cases}
                    \frac14|v(0)|^\alpha
                    \alpha^\alpha h(\frac12) \bar{\tau} & \text{ for the map \eqref{eq:LSV}} \; ,
                    \\[.75ex]
                   \frac19 \bigl( |v(0)|^\alpha + |v(1)|^\alpha \bigr)
                   \alpha^\alpha h(\frac13) \bar{\tau} & \text{ for the map \eqref{eq:PM}} \; .
                \end{cases}
            \]
            Here $\bar{\tau} = \int_Z \tau \mrd \mu_Z$.
        \item\label{lem:lsv:inter} 
            $V - \E(V \mid \cP) \in L^p$ for some $p > \alpha$.
        \item\label{lem:lsv:V*} 
                $n^{-1/\alpha} \max_{0\le k<n}  V^* \circ F^k  \to_w 0$
                on $(Z, \mu_Z)$.
    \end{enumerate}
\end{lemma}

%

\begin{proof}
    We give the details for the map~\eqref{eq:PM}. The details for the map~\eqref{eq:LSV}
    are similar and simpler.

    Let $a_1 = \frac13$ and $a_k = a_{k+1}(1 + (3 a_{k+1})^{1/\alpha})$, $k \geq 1$.
    By a standard calculation, see for example~\cite{H05},
$a_k \sim\frac13 \alpha^\alpha k^{-\alpha}$.
    Let $z_k = \frac13(a_k + 1) $ and $z'_k = 1-z_k$.
    The partition $\cP$ consists of the intervals $(z_{k}, z_{k-1})$ and $(z'_{k-1}, z'_{k})$,
    $k \geq 2$, on which $\tau$ equals $k$, and $(z_1, z'_1)$ where $\tau$ equals~$1$.
    
    Observe that $F = f^\tau$ has full branches,
    i.e.\ $Fa = Z$ for every $a \in \cP$, modulo zero measure.
    It is standard that the unique $F$-invariant absolutely
    continuous measure $\mu_Z$ has continuous density $h_Z$
    bounded away from zero (see for example~\cite[Proposition~2.5]{KKM19exp}).
    Moreover, $h$ is bounded below and
    $h|_Z = h_Z / \bar{\tau}$.

    If $z \in (\frac13, z_k)$ and $0 < \ell \leq k$, then
    $f^\ell z \in (0, a_{k - \ell + 1})$, so
    $|f^\ell z| \lesssim (k - \ell)^{-\alpha}$.
    Similarly, if $z \in (z'_k, \frac23)$, then
    $|1-f^\ell z| \lesssim (k - \ell)^{-\alpha}$.
    Let $\theta\in(0,1]$ be the H\"older exponent of~$v$.
    Without loss, we assume that $\theta < 1 / \alpha$.
    Define $\hat v = v(0) 1_{(\frac13, \frac12)} + v(1) 1_{(\frac12,\frac23)}$ on $Z$.
    Then
    \begin{equation} \label{eq:hatv}
        \Bigl| \ell \hat v(z) - \sum_{j=0}^{\ell - 1} v(f^j z) \Bigr|
        \leq |\hat v(z) - v(z)| + \sum_{j=1}^{\tau(z) - 1} |\hat v(z) - v(f^j z)|
        \lesssim \tau(z)^{\beta}
    \end{equation}
    for $\ell \leq \tau(z)$,
    where $\beta = 1-\alpha \theta \in (0,1)$.
    In particular,
$|\tau \hat v -V| \lesssim \tau^\beta$.

By symmetry and continuity of $h_Z$,
\[
    {\SMALL \mu_Z(z>\frac12,\,\tau > k)=
        \mu_Z(z<\frac12,\,\tau > k)
    =\mu_Z((\frac13, z_k))}\sim \frac{h_Z(\frac13) \alpha^\alpha}{9 k^\alpha}\;.
\]

Let $B$ be a Borel set in $\bbS^{d-1}$ and suppose that $v(0)/|v(0)|\in B$,
$v(1)/|v(1)|\not\in B$.
Then
\begin{align*}
\frac{\mu_Z(|\tau\hat v|>rt,\,\tau\hat v/|\tau\hat v|\in B)}{\mu_Z(|\tau\hat v|>t)}
& =\frac{\mu_Z(z<\frac12,\,\tau>rt/|v(0)|)}{\mu_Z(z<\frac12,\,\tau>t/|v(0)|)+
\mu_Z(z>\frac12,\,\tau>t/|v(1)|)} \\
& \to r^{-\alpha}\frac{|v(0)|^\alpha}{|v(0)|^\alpha+|v(1)|^\alpha}
\quad\text{as $t\to\infty$}
\; .
\end{align*}
The calculations for the remaining Borel sets $B$ are similar, and it follows that 
    $\tau \hat v$ is regularly varying with index $\alpha$ 
    and that the probability measure $\sigma$ as in Definition~\ref{def:reg} is given by the formula in part~\ref{lem:lsv:reg}.
By~\eqref{eq:hatv}, $V$ is regularly varying with index $\alpha$ and the same $\sigma$, proving 
 part~\ref{lem:lsv:reg}.

Moreover,
$\mu_Z(|\tau\hat v|>n)\sim cn^{-\alpha}$
with $c$ as in part~\ref{lem:lsv:c}, so
$\mu_Z(|V|>n)\sim cn^{-\alpha}$ by~\eqref{eq:hatv}. Part~\ref{lem:lsv:c} follows by Remark~\ref{rmk:weakZ}\ref{rmk:weakZ:reg}.

It is immediate from~\eqref{eq:hatv} that
                $|V(z) - V(z')| \lesssim \tau(a)^\beta$
            for all $z,z'\in a$, $a\in\cP$.
 Part~\ref{lem:lsv:inter} follows by Remark~\ref{rmk:weakZ}\ref{rmk:weakZ:Lp}.

    Finally, it follows from~\eqref{eq:hatv} that $V^* \lesssim \tau^\beta$,
    from which $V^* \in L^q (\mu_Z)$ for some $q > \alpha$, and
    \begin{align*}
        \int \Bigl( n^{-1/\alpha} \max_{0\le k<n}  V^* \circ F^k  \Bigr)^q \mrd \mu_Z
        & \leq n^{-q / \alpha} \sum_{k < n} \int (V^*)^q \circ F^k \mrd \mu_Z
      = n^{-q / \alpha + 1} \|V^*\|_q^q
        \to 0
        \; .
    \end{align*}
    This proves~\ref{lem:lsv:V*} and completes the proof of the lemma.
\end{proof}

Theorems~\ref{thm:PM} and~\ref{thm:tight} now follow from
Theorems~\ref{thm:Y} and~\ref{thm:Yp}.  Moreover, $L_\alpha$ is identified as the $\alpha$-stable L\'evy process with spectral measure $\Lambda
        = c \cos \frac{\pi \alpha}{2} \Gamma(1-\alpha) \sigma/\bar\tau$
        with $c$ and $\sigma$ as in Lemma~\ref{lem:lsv}.

Finally, as a consequence of these results combined with Theorem~\ref{thm:rig}, we
can record the desired conclusion for homogenisation of fast-slow
systems with fast dynamics given by one of the intermittent maps in
Section~\ref{sec:intro}.

\begin{cor} \label{cor:PM}
    Consider the intermittent map~\eqref{eq:LSV} or~\eqref{eq:PM} with $\alpha\in(1,2)$
    and let $v\colon Y\to\R^d$ be H\"older with $\int_Y v\mrd\mu=0$ and $v(0)\neq0$,
    also $v(1)\neq0$ in case of~\eqref{eq:PM}.

Consider the fast-slow system~\eqref{eq:fs}
with initial condition $x^{(n)}_0 = \xi_n$ such that $\lim_{n\to\infty}\xi_n = \xi$.
Suppose that 
$a\in C^{\beta}(\R^m,\R^m)$, $b \in C^\gamma(\R^m,\R^{m\times d})$
for some $\beta>1$, $\gamma > \alpha$.
Define $W_n$ as in~\eqref{eq:Wn} and $X_n(t) = x_{\floor{nt}}^{(n)}$.
Let $\PP$ be any probability measure on $Y$ that is absolutely continuous with respect to Lebesgue, and regard $W_n$ and $X_n$ as processes on $(Y,\PP)$.

    Let $\ell_{k}$ denote the linear path function on $\R^k$ and
    let $\phi_{b}$ be the path function on $\R^{d+m}$
    as in Definition~\ref{def:phi_b_def}.
Fix $p>\alpha$.
Then
    \[
        ((W_n, X_n), \ell_{d+m}) \to_w ((L_\alpha, X), \phi_{b})
        \qquad \text{as} \qquad n \to \infty
    \]
in $(\sD^{p\var}([0,1],\R^{d +m}), \balpha_{p\var})$, where 
$L_\alpha$ is the $\alpha$-stable L\'evy process with
spectral measure $\Lambda=
c \cos \frac{\pi \alpha}{2} \Gamma(1-\alpha) \sigma/\bar\tau$
        with $c$ and $\sigma$ as in Lemma~\ref{lem:lsv},
  and $X$ is the solution of the Marcus differential equation~\eqref{eq:Marcus}.
  \qed
\end{cor}

\end{document}